\documentclass[12pt]{amsart}
\usepackage{amssymb}
\begin{document}
\theoremstyle{plain}
\newtheorem{thm}{Theorem}[section]
\newtheorem*{thm1}{Theorem 1}
\newtheorem*{thm2}{Theorem 2}
\newtheorem{lemma}[thm]{Lemma}
\newtheorem{lem}[thm]{Lemma}
\newtheorem{cor}[thm]{Corollary}
\newtheorem{pro}[thm]{Proposition}
\newtheorem{propose}[thm]{Proposition}
\newtheorem{variant}[thm]{Variant}
\theoremstyle{definition}
\newtheorem{notations}[thm]{Notations}
\newtheorem{rem}[thm]{Remark}
\newtheorem{rmk}[thm]{Remark}
\newtheorem{rmks}[thm]{Remarks}
\newtheorem{defi}[thm]{Definition}
\newtheorem{exe}[thm]{Example}
\newtheorem{claim}[thm]{Claim}
\newtheorem{ass}[thm]{Assumption}
\newtheorem{prodefi}[thm]{Proposition-Definition}
\newtheorem{que}[thm]{Question}
\numberwithin{equation}{section}
\newcounter{elno}                
\def\points{\list
{\hss\llap{\upshape{(\roman{elno})}}}{\usecounter{elno}}}
\let\endpoints=\endlist


\catcode`\@=11
%
%
\def\opn#1#2{\def#1{\mathop{\kern0pt\fam0#2}\nolimits}} 
\def\bold#1{{\bf #1}}%
\def\underrightarrow{\mathpalette\underrightarrow@}
\def\underrightarrow@#1#2{\vtop{\ialign{$##$\cr
 \hfil#1#2\hfil\cr\noalign{\nointerlineskip}%
 #1{-}\mkern-6mu\cleaders\hbox{$#1\mkern-2mu{-}\mkern-2mu$}\hfill
 \mkern-6mu{\to}\cr}}}
\let\underarrow\underrightarrow
\def\underleftarrow{\mathpalette\underleftarrow@}
\def\underleftarrow@#1#2{\vtop{\ialign{$##$\cr
 \hfil#1#2\hfil\cr\noalign{\nointerlineskip}#1{\leftarrow}\mkern-6mu
 \cleaders\hbox{$#1\mkern-2mu{-}\mkern-2mu$}\hfill
 \mkern-6mu{-}\cr}}}
%
%

%
\def\:{\colon}
\let\oldtilde=\tilde
\def\tilde#1{\mathchoice{\widetilde{#1}}{\widetilde{#1}}%
{\indextil{#1}}{\oldtilde{#1}}}
\def\indextil#1{\lower2pt\hbox{$\textstyle{\oldtilde{\raise2pt%
\hbox{$\scriptstyle{#1}$}}}$}}
\def\pnt{{\raise1.1pt\hbox{$\textstyle.$}}}
%

%
\let\amp@rs@nd@\relax
\newdimen\ex@
\ex@.2326ex
\newdimen\bigaw@
\newdimen\minaw@
\minaw@16.08739\ex@
\newdimen\minCDaw@
\minCDaw@2.5pc
\newif\ifCD@
\def\minCDarrowwidth#1{\minCDaw@#1}
\newenvironment{CD}{\@CD}{\@endCD}
\def\@CD{\def\A##1A##2A{\llap{$\vcenter{\hbox
 {$\scriptstyle##1$}}$}\Big\uparrow\rlap{$\vcenter{\hbox{%
$\scriptstyle##2$}}$}&&}%
\def\V##1V##2V{\llap{$\vcenter{\hbox
 {$\scriptstyle##1$}}$}\Big\downarrow\rlap{$\vcenter{\hbox{%
$\scriptstyle##2$}}$}&&}%
\def\={&\hskip.5em\mathrel
 {\vbox{\hrule width\minCDaw@\vskip3\ex@\hrule width
 \minCDaw@}}\hskip.5em&}%
\def\verteq{\Big\Vert&&}%
\def\noarr{&&}%
\def\vspace##1{\noalign{\vskip##1\relax}}\relax\let\amp@rs@nd@&\iffalse}\fi
 \CD@true\vcenter\bgroup\relax\let\\=\cr\iffalse}\fi\tabskip\z@skip\baselineskip20\ex@
 \lineskip3\ex@\lineskiplimit3\ex@\halign\bgroup
 &\hfill$\m@th##$\hfill\cr}
\def\@endCD{\cr\egroup\egroup}
%
\def\>#1>#2>{\amp@rs@nd@\setbox\z@\hbox{$\scriptstyle
 \;{#1}\;\;$}\setbox\@ne\hbox{$\scriptstyle\;{#2}\;\;$}\setbox\tw@
 \hbox{$#2$}\ifCD@
 \global\bigaw@\minCDaw@\else\global\bigaw@\minaw@\fi
 \ifdim\wd\z@>\bigaw@\global\bigaw@\wd\z@\fi
 \ifdim\wd\@ne>\bigaw@\global\bigaw@\wd\@ne\fi
 \ifCD@\hskip.5em\fi
 \ifdim\wd\tw@>\z@
 \mathrel{\mathop{\hbox to\bigaw@{\rightarrowfill}}\limits^{#1}_{#2}}\else
 \mathrel{\mathop{\hbox to\bigaw@{\rightarrowfill}}\limits^{#1}}\fi
 \ifCD@\hskip.5em\fi\amp@rs@nd@}
\def\<#1<#2<{\amp@rs@nd@\setbox\z@\hbox{$\scriptstyle
 \;\;{#1}\;$}\setbox\@ne\hbox{$\scriptstyle\;\;{#2}\;$}\setbox\tw@
 \hbox{$#2$}\ifCD@
 \global\bigaw@\minCDaw@\else\global\bigaw@\minaw@\fi
 \ifdim\wd\z@>\bigaw@\global\bigaw@\wd\z@\fi
 \ifdim\wd\@ne>\bigaw@\global\bigaw@\wd\@ne\fi
 \ifCD@\hskip.5em\fi
 \ifdim\wd\tw@>\z@
 \mathrel{\mathop{\hbox to\bigaw@{\leftarrowfill}}\limits^{#1}_{#2}}\else
 \mathrel{\mathop{\hbox to\bigaw@{\leftarrowfill}}\limits^{#1}}\fi
 \ifCD@\hskip.5em\fi\amp@rs@nd@}
%
%
\newenvironment{CDS}{\@CDS}{\@endCDS}
\def\@CDS{\def\A##1A##2A{\llap{$\vcenter{\hbox
 {$\scriptstyle##1$}}$}\Big\uparrow\rlap{$\vcenter{\hbox{%
$\scriptstyle##2$}}$}&}%
\def\V##1V##2V{\llap{$\vcenter{\hbox
 {$\scriptstyle##1$}}$}\Big\downarrow\rlap{$\vcenter{\hbox{%
$\scriptstyle##2$}}$}&}%
\def\={&\hskip.5em\mathrel
 {\vbox{\hrule width\minCDaw@\vskip3\ex@\hrule width
 \minCDaw@}}\hskip.5em&}
\def\verteq{\Big\Vert&}
\def\novarr{&}
\def\noharr{&&}
\def\SE##1E##2E{\slantedarrow(0,18)(4,-3){##1}{##2}&}
\def\SW##1W##2W{\slantedarrow(24,18)(-4,-3){##1}{##2}&}
\def\NE##1E##2E{\slantedarrow(0,0)(4,3){##1}{##2}&}
\def\NW##1W##2W{\slantedarrow(24,0)(-4,3){##1}{##2}&}
\def\slantedarrow(##1)(##2)##3##4{%
\thinlines\unitlength1pt\lower 6.5pt\hbox{\begin{picture}(24,18)%
\put(##1){\vector(##2){24}}%
\put(0,8){$\scriptstyle##3$}%
\put(20,8){$\scriptstyle##4$}%
\end{picture}}}
\def\vspace##1{\noalign{\vskip##1\relax}}\relax\let\amp@rs@nd@&\iffalse}\fi
 \CD@true\vcenter\bgroup\relax\let\\=\cr\iffalse}\fi\tabskip\z@skip\baselineskip20\ex@
 \lineskip3\ex@\lineskiplimit3\ex@\halign\bgroup
 &\hfill$\m@th##$\hfill\cr}
\def\@endCDS{\cr\egroup\egroup}
%
\newdimen\TriCDarrw@
\newif\ifTriV@
\newenvironment{TriCDV}{\@TriCDV}{\@endTriCD}
\newenvironment{TriCDA}{\@TriCDA}{\@endTriCD}
\def\@TriCDV{\TriV@true\def\TriCDpos@{6}\@TriCD}
\def\@TriCDA{\TriV@false\def\TriCDpos@{10}\@TriCD}
\def\@TriCD#1#2#3#4#5#6{%
\setbox0\hbox{$\ifTriV@#6\else#1\fi$}
\TriCDarrw@=\wd0 \advance\TriCDarrw@ 24pt
\advance\TriCDarrw@ -1em
\def\SE##1E##2E{\slantedarrow(0,18)(2,-3){##1}{##2}&}
\def\SW##1W##2W{\slantedarrow(12,18)(-2,-3){##1}{##2}&}
\def\NE##1E##2E{\slantedarrow(0,0)(2,3){##1}{##2}&}
\def\NW##1W##2W{\slantedarrow(12,0)(-2,3){##1}{##2}&}
\def\slantedarrow(##1)(##2)##3##4{\thinlines\unitlength1pt
\lower 6.5pt\hbox{\begin{picture}(12,18)%
\put(##1){\vector(##2){12}}%
\put(-4,\TriCDpos@){$\scriptstyle##3$}%
\put(12,\TriCDpos@){$\scriptstyle##4$}%
\end{picture}}}
\def\={\mathrel {\vbox{\hrule
   width\TriCDarrw@\vskip3\ex@\hrule width
   \TriCDarrw@}}}
\def\>##1>>{\setbox\z@\hbox{$\scriptstyle
 \;{##1}\;\;$}\global\bigaw@\TriCDarrw@
 \ifdim\wd\z@>\bigaw@\global\bigaw@\wd\z@\fi
 \hskip.5em
 \mathrel{\mathop{\hbox to \TriCDarrw@
{\rightarrowfill}}\limits^{##1}}
 \hskip.5em}
\def\<##1<<{\setbox\z@\hbox{$\scriptstyle
 \;{##1}\;\;$}\global\bigaw@\TriCDarrw@
 \ifdim\wd\z@>\bigaw@\global\bigaw@\wd\z@\fi
 \mathrel{\mathop{\hbox to\bigaw@{\leftarrowfill}}\limits^{##1}}
 }
 \CD@true\vcenter\bgroup\relax\let\\=\cr\iffalse}\fi
 \tabskip\z@skip\baselineskip20\ex@
 \lineskip3\ex@\lineskiplimit3\ex@
 \ifTriV@
 \halign\bgroup
 &\hfill$\m@th##$\hfill\cr
#1&\multispan3\hfill$#2$\hfill&#3\\
&#4&#5\\
&&#6\cr\egroup%
\else
 \halign\bgroup
 &\hfill$\m@th##$\hfill\cr
&&#1\\%
&#2&#3\\
#4&\multispan3\hfill$#5$\hfill&#6\cr\egroup
\fi}
\def\@endTriCD{\egroup}
\newcommand{\mc}{\mathcal}
\newcommand{\mb}{\mathbb}
\newcommand{\surj}{\twoheadrightarrow}
\newcommand{\inj}{\hookrightarrow}
\newcommand{\zar}{{\rm zar}}
\newcommand{\an}{{\rm an}}
\newcommand{\red}{{\rm red}}
\newcommand{\codim}{{\rm codim}}
\newcommand{\rank}{{\rm rank}}
\newcommand{\Ker}{{\rm Ker \ }}
\newcommand{\Pic}{{\rm Pic}}
\newcommand{\Div}{{\rm Div}}
\newcommand{\Hom}{{\rm Hom}}
\newcommand{\im}{{\rm im}}
\newcommand{\Spec}{{\rm Spec \,}}
\newcommand{\Frac}{{\rm Frac \,}}
\newcommand{\Sing}{{\rm Sing}}
\newcommand{\sing}{{\rm sing}}
\newcommand{\reg}{{\rm reg}}
\newcommand{\Char}{{\rm char}}
\newcommand{\Tr}{{\rm Tr}}
\newcommand{\ord}{{\rm ord}}
\newcommand{\id}{{\rm id}}
\newcommand{\Gal}{{\rm Gal}}
\newcommand{\Min}{{\rm Min \ }}
\newcommand{\Max}{{\rm Max \ }}
\newcommand{\Alb}{{\rm Alb}\,}
\newcommand{\GL}{{\rm GL}\,}        
\newcommand{\PGL}{{\rm PGL}\,}
\newcommand{\Bir}{{\rm Bir}}
\newcommand{\Aut}{{\rm Aut}}
\newcommand{\End}{{\rm End}}
\newcommand{\Per}{{\rm Per}\,}
\newcommand{\ie}{{\it i.e.\/},\ }
\newcommand{\niso}{\not\cong}
\newcommand{\nin}{\not\in}
\newcommand{\soplus}[1]{\stackrel{#1}{\oplus}}
\newcommand{\by}[1]{\stackrel{#1}{\rightarrow}}
\newcommand{\longby}[1]{\stackrel{#1}{\longrightarrow}}
\newcommand{\vlongby}[1]{\stackrel{#1}{\mbox{\large{$\longrightarrow$}}}}
\newcommand{\ldownarrow}{\mbox{\Large{\Large{$\downarrow$}}}}
\newcommand{\lsearrow}{\mbox{\Large{$\searrow$}}}
\renewcommand{\d}{\stackrel{\mbox{\scriptsize{$\bullet$}}}{}}
\newcommand{\dlog}{{\rm dlog}\,}    
\newcommand{\longto}{\longrightarrow}
\newcommand{\vlongto}{\mbox{{\Large{$\longto$}}}}
\newcommand{\limdir}[1]{{\displaystyle{\mathop{\rm lim}_{\buildrel\longrightarrow\over{#1}}}}\,}
\newcommand{\liminv}[1]{{\displaystyle{\mathop{\rm lim}_{\buildrel\longleftarrow\over{#1}}}}\,}
\newcommand{\norm}[1]{\mbox{$\parallel{#1}\parallel$}}
\newcommand{\boxtensor}{{\Box\kern-9.03pt\raise1.42pt\hbox{$\times$}}}
\newcommand{\into}{\hookrightarrow}
\newcommand{\image}{{\rm image}\,}
\newcommand{\Lie}{{\rm Lie}\,}      
\newcommand{\CM}{\rm CM}
\newcommand{\sext}{\mbox{${\mathcal E}xt\,$}}  
\newcommand{\shom}{\mbox{${\mathcal H}om\,$}}  
\newcommand{\coker}{{\rm coker}\,}  
\newcommand{\sm}{{\rm sm}}
\newcommand{\pgcd}{\text{pgcd}}
\newcommand{\tensor}{\otimes}
\renewcommand{\iff}{\mbox{ $\Longleftrightarrow$ }}
\newcommand{\supp}{{\rm supp}\,}
\newcommand{\ext}[1]{\stackrel{#1}{\wedge}}
\newcommand{\onto}{\mbox{$\,\>>>\hspace{-.5cm}\to\hspace{.15cm}$}}
\newcommand{\propsubset}
{\mbox{$\textstyle{
\subseteq_{\kern-5pt\raise-1pt\hbox{\mbox{\tiny{$/$}}}}}$}}
\newcommand{\sA}{{\mathcal A}}
\newcommand{\sB}{{\mathcal B}}
\newcommand{\sC}{{\mathcal C}}
\newcommand{\sD}{{\mathcal D}}
\newcommand{\sE}{{\mathcal E}}
\newcommand{\sF}{{\mathcal F}}
\newcommand{\sG}{{\mathcal G}}
\newcommand{\sH}{{\mathcal H}}
\newcommand{\sI}{{\mathcal I}}
\newcommand{\sJ}{{\mathcal J}}
\newcommand{\sK}{{\mathcal K}}
\newcommand{\sL}{{\mathcal L}}
\newcommand{\sM}{{\mathcal M}}
\newcommand{\sN}{{\mathcal N}}
\newcommand{\sO}{{\mathcal O}}
\newcommand{\sP}{{\mathcal P}}
\newcommand{\sQ}{{\mathcal Q}}
\newcommand{\sR}{{\mathcal R}}
\newcommand{\sS}{{\mathcal S}}
\newcommand{\sT}{{\mathcal T}}
\newcommand{\sU}{{\mathcal U}}
\newcommand{\sV}{{\mathcal V}}
\newcommand{\sW}{{\mathcal W}}
\newcommand{\sX}{{\mathcal X}}
\newcommand{\sY}{{\mathcal Y}}
\newcommand{\sZ}{{\mathcal Z}}
\newcommand{\A}{{\mathbb A}}
\newcommand{\B}{{\mathbb B}}
\newcommand{\C}{{\mathbb C}}
\newcommand{\D}{{\mathbb D}}
\newcommand{\E}{{\mathbb E}}
\newcommand{\F}{{\mathbb F}}
\newcommand{\G}{{\mathbb G}}
\newcommand{\HH}{{\mathbb H}}
\newcommand{\I}{{\mathbb I}}
\newcommand{\J}{{\mathbb J}}
\newcommand{\M}{{\mathbb M}}
\newcommand{\N}{{\mathbb N}}
\renewcommand{\P}{{\mathbb P}}
\newcommand{\Q}{{\mathbb Q}}
\newcommand{\R}{{\mathbb R}}
\newcommand{\T}{{\mathbb T}}
\newcommand{\U}{{\mathbb U}}
\newcommand{\V}{{\mathbb V}}
\newcommand{\W}{{\mathbb W}}
\newcommand{\X}{{\mathbb X}}
\newcommand{\Y}{{\mathbb Y}}
\newcommand{\Z}{{\mathbb Z}}

\newcommand{\fix}{\mathrm{Fix}}

\title[]{Periodic points of birational maps on projective surfaces }

\author{Xie Junyi}
\address{Centre de Math\'ematiques Laurent Schwartz \'Ecole Polytechnique, 91128, Palaiseau
Cedex, France}

\email{jxie@clipper.ens.fr}

\date{\today}

\maketitle

\begin{abstract}
We classify birational maps of projective smooth surfaces  whose non-critical periodic points are Zariski dense. In particular, we show that if the first dynamical degree is greater than one, then the periodic points are Zariski dense.
\end{abstract}

\section{Introduction}
\bibliographystyle{plain}
Hrushovski and Fakhruddin \cite{fa} recently proved by purely algebraic methods that the set of periodic points of a polarized endomorphism of degree at least $2$ of a projective variety over any algebraically closed field is Zariski dense. In this article, we give a complete classification of birational surface maps whose non-critical points are Zariski dense.

In order to state our main result, we need to recall some basic notions related to birational maps of surfaces.  Let $X$ be a projective surface, $L$ be an ample line bundle on $X$ and $f:X\dashrightarrow X$ be any birational map.  We set $\deg_L(f)=(f^*L\cdot L)$ and call it the degree of $f$ respect to $L$. One can show (\cite{boucksomfavrejonsson}) that $\deg_L (f^{m+n})\leq 2\deg_L(f^m)\deg_L(f^n)$ for all $n,m \geq 0$, so that the limit  $$\lambda_1(f):=\lim_{n \rightarrow \infty}\deg_L(f^n)^{1/n}\geq 1$$ is well defined. It is independent on the choice of $L$ and we call it the first dynamical degree of $f$. It is constant on the conjugacy class of $f$ in the group of birational transformations of $X$.
\begin{thm}\label{class}
Let $X$ be a projective smooth surface over an algebraically closed field of characteristic different from $2$ and $3$, $L$ be an ample line bundle and $f$ be a birational map of $X$. We denote by $\mathcal{P}$ the set of non-critical periodic points of $f$. Then we are in one of the following  four cases.
\begin{points}
\item If $\lambda_1(f)>1$, then $\mathcal{P}$ is Zariski dense.
\item If $\lambda_1(f)=1$ and $\deg(f^n)\sim cn,$ where $c>0$, then $\mathcal{P}$ is Zariski dense if and only if for some $N>0$, $f^N$ preserves a rational fibration and acts on the base as identity.

\item If $\lambda_1(f)=1$ and $\deg(f^n)\sim cn^2,$ where $c>0$, then $\mathcal{P}$ is Zariski dense if and only if for some $N>0$, $f^N$ preserves an elliptic fibration and acts on the base as identity.
\item If $\lambda_1(f)=1$, and $\deg(f^n)$ is bounded, then $\mathcal{P}$ is Zariski dense if and only if there is an integer $N>0$ such that $f^N=\id$.

\end{points}
\end{thm}

The most interesting case in the previous theorem
 is case (i). We actually prove this result over a field of arbitrary characteristic.

\begin{thm}\label{main1}
Let $X$ be a projective surface over an algebraically closed field $\mathbf{k}$, and $f:X\dashrightarrow X$ be a
birational map. If
$\lambda_1(f)>1$ then the set of non-critical periodic points is
Zariski dense.

\end{thm}

In the case $\mathbf{k}=\mathbb{C}$, this theorem has been proved in many cases using analytic methods. In \cite{Diller2010} of Diller, Dujardin and Guedj, they proved it for birational polynomial maps; or more generally for any birational map such that the points of indeterminacy of $f^{-1}$ does not cluster too much near the points of indeterminacy of $f$. We refer to  \cite{Bedford2005} for a precise statement.

It follows from \cite{favre} that any birational map with $\lambda_1>1$ defined over a non-rational
surface is birationally equivalent to an automorphism.
When $f$ is an automorphism with $\lambda_1>1$, then it is possible to get a more precise count
on the number of isolated periodic points based on Saito's fixed point formula, see
\cite{Iwasaki2010,Saito1987}.
\begin{thm}\label{cisop}Let $X$ be a projective smooth surface over a field with characteristic $0$, $f:X\rightarrow X$ be a nontrivial automorphism with $\lambda_1(f)>1$. We denote by $\#\Per_n(f)$ the number of isolated periodic points of period $n$ counted with multiplicities. Then  we have $$\#\Per_n(f)\sim \lambda_1(f)^n.$$\end{thm}

For completeness, we also induce the following corollary of the work of Amerik in \cite{Amerik}.
\begin{thm}\label{coram1}Let $X$ be a surface over an algebraically closed field of characteristic $0$, $f:X\dashrightarrow X$ be a birational map with $\lambda_1(f)>1$. Then there is a point $x\in X$ such that $f^n(x)\in X-I(f)$ for any $n\in \mathbb{Z}$ and $\{f^{n}(x)|n\in \mathbb{Z}\}$ is Zariski dense.

\end{thm}

Let us explain our strategy to prove Theorem \ref{main1}.
We follow the original method of Hrushovski and Fakhruddin by reducing our result to the case of finite fields.

For simplicity, we may assume that $X=\mathbb{P}^2$ and $f=[f_0,f_1,f_2]$ is a birational map with $\lambda_1(f)>1$ and has integral coefficients.

First assume we can find a prime $p>0$ such that the reduction $f_p$ modulo $p$ of $f$ satisfies $\lambda_1(f_p)>1$. Then $\mathcal{P}(f_p)$ is Zariski dense in $\mathbb{P}^2(\overline{\mathbb{F}_p})$ by a direct application of Hrushovski's arguments. One then lifts these periodic points to $\mathbb{P}^2(\overline{\mathbb{Q}})$ by combining a result of Cantat \cite{Invarianthypersurfaces} proving that most periodic points are isolated together with a simple dimensional argument borrowed from Fakhruddin \cite{fa}.

The main difficulty thus lies in proving that $\lambda_1(f_p)>1$ for at least one prime $p$. Actually we show:
 \begin{thm}\label{p1} Let $f:\mathbb{P}^2(\mathbb{Z})\dashrightarrow\mathbb{P}^2(\mathbb{Z})$ be a
birational map defined over $\mathbb{Z}$ (i.e. $f=[f_0,f_1,f_2]$
such that all coefficients of $f_i$'s are integers). Then for any
prime $p$ sufficiently large, $f$ induces a birational map
$f_p:\mathbb{P}^2(\mathbb{F}_p)\dashrightarrow\mathbb{P}^2(\mathbb{F}_p)$
on the special fiber at $p$, and
$$\lim_{p\rightarrow\infty}\lambda_1(f_p)=\lambda_1(f).$$

 \end{thm}
We also give an example of a birational map $f$ on $\mathbb{P}^2(\mathbb{Z})$ such that  $\lambda_1(f_p)<\lambda_1(f)$ for all $p$ in section \ref{example}.
\\

In fact we prove a quite general version of Theorem \ref{p1} for families of birational maps of surfaces over integral schemes. It allows us to prove:
\begin{thm}\label{1general}Let $\mathbf{k}$ be an algebraically closed field and $d\geq 2$ be an integer. We denote by $\Bir_d$ the space of birational maps of $\mathbb{P}^2(\mathbf{k})$ of degree $d$. Then for any $\lambda<d$, $U_{\lambda}=\{f\in \Bir_d | \lambda_1(f)>\lambda\}$ is a Zariski dense open set of $\Bir_d$.

In particular, for a general birational map $f$ of degree $d>1$, we have $\lambda_1(f)>1$.

\end{thm}

In order to prove Theorem \ref{p1}, we need to  control $\lambda_1(f)$ in terms of the degree of a fixed iterate of $f$. The control is given by our next theorem.
\begin{thm}\label{f^2>2^18f1}
Let $X$ be a projective surface over an algebraically closed field, $L$ be an ample bundle on $X$, $f:X\dashrightarrow X$ be a birational map
and $q=\frac{\deg_{L}(f^2)}{3^{18}\sqrt{2}\deg_{L}(f)} $. If $q\geq1$, then $$\lambda_1(f)>\frac{(4\times3^{36}-1)q^2+1}{4\times3^{36}q}\geq1.$$

\end{thm}
In particular if $\deg_L(f^2)\geq 3^{18}\sqrt{2}\deg_L(f)$ then $\lambda_1(f)>1.$
 This result has been stated in \cite{favreBourbaki} by Favre without proof.

The key tool to prove this theorem is the action of the birational map as an isometry on a suitable hyperbolic space of infinite dimension. This space is constructed as a set of classes in the Riemann-Zariski space of a surface which was introduced by Cantat in \cite{Cantat2007}. See also \cite{Cantat}, \cite{boucksomfavrejonsson} and \cite{favreBourbaki}.\\

The article is organized in 7 sections.

In section \ref{notation} we introduce backgrand on intersection theory on surface and  Riemann-Zariski space that will be used in this article.
In section \ref{deg} we prove Theorem \ref{f^2>2^18f1}
In section \ref{variation} we apply Theorem \ref{main1} to study the behavior of the first dynamical degree in a family of birational maps of surfaces.
In section \ref{example} we give an example of a birational map $f$ on $\mathbb{P}^2(\mathbb{Z})$ such that  $\lambda_1(f_p)<\lambda_1(f)$ for all prime $p$.
In section \ref{case1}, we prove Theorem \ref{main1} and Theorem \ref{coram1}.
In section \ref{Automorphisms>1}, we use Saito's formula to study isolated periodic points of automorphisms and prove Theorem \ref{cisop}.
In section \ref{case2}, we study the Zariski density of periodic points in the case $\lambda_1=1$.
We use Theorem \ref{fadi} of Favre and Diller \cite{favre} which divides the birational map of surfaces into four cases by the degree growth.
We analyze all these cases in Theorem \ref{fadi}.
Finally we combine the results in section \ref{case1} with section \ref{case2} to get Theorem \ref{class}.

\section*{Acknowledgement}
I thank ANR Berko for the support. I thank Charles Favre for his direction during the writing of this article, and the referee for his careful reading. I thank Fu Lie with whom I had several interesting discussions about the material presented here. Also, I thank  Joseph H.Silverman for pointing out those mistakes in the first version of this article. Finally I thank Zhuang Wei Dong for his help in correcting my English.

\section{Background and Notation }\label{notation}
In this paper, a variety is always defined over an algebraically
closed field and we use the notation $\mathbf{k}$ to denote an algebraically closed field of arbitrary characteristic except in Subsection \ref{subam}, Section \ref{Automorphisms>1} and Section \ref{case2}. In Subsection \ref{subam} and Section \ref{Automorphisms>1}, $\mathbf{k}$  denotes an algebraically closed field of characteristic $0$. In Section \ref{case2}, $\mathbf{k}$  denotes an algebraically closed field of characteristic different from $2$ and $3$.

\subsection{N\'eron-Severi group }
In this subsection, we will recall the definition and some properties of the N\'eron-Severi group. All the material in this section can be found in \cite{Lazarsfeld}.

Let $X$ be a projective variety over $\mathbf{k}$. We denote by $\Pic(X)$ the Picard group of $X$. We denote by $\Pic^0(X)$ the group of elements in $\Pic(X)$ which are numerically equivalent to $0$.  We recall that it is an abelian variety.
The N\'eron-Severi group of $X$ is defined as the group of numerical
equivalence classes of divisors of $X$. We denote it by
$N^1(X)$, and write
$N^1(X)_{\mathbb{R}}=N^1(X)\otimes_{\mathbb{Z}}\mathbb{R}$.
This group $N^1(X)$ is a free abelian group of finite
rank (see \cite{n1}).
Let $\phi:X\rightarrow Y$ be a morphism between projective variety
over an algebraically closed field. Then it induces natural maps
$\phi^*: N^1(Y)_{\mathbb{R}}\rightarrow N^1(X)_{\mathbb{R}}$, and
$\phi_*: N^1(X)_{\mathbb{R}}\rightarrow N^1(Y)_{\mathbb{R}}$ induced
by the actions on divisors .

\begin{pro}(Pull-back formula)
 Let $ \pi : Y \rightarrow X $ be a
surjective morphism between projective varieties over an
algebraically closed field. For any choose $\alpha_1,\cdots , \alpha_r\in N^1(Y)$ ,we
have $(\pi^*\alpha_1 \cdot\ldots\cdot\pi^*\alpha_r) = \deg (\pi)(\alpha_1
\cdot\ldots\cdot \alpha_r)$.
\end{pro}

One defines $N_1(X)_{\mathbb{R}}$ to be the real space of numerical
equivalence classes of real one-cycles of $X$. One has a perfect pairing$$N^1(X)_{\mathbb{R}}\times
N_1(X)_{\mathbb{R}}\rightarrow \mathbb{R},\,\,\,
(\delta,\gamma)\rightarrow (\delta\cdot\gamma)\in \mathbb{R}.$$

So that $N_1(X)_{\mathbb{R}}$ is dual to $N^1(X)_{\mathbb{R}}$.
In particular $N_1(X)_{\mathbb{R}}$ is a finite dimensional real
vector space on which we put the standard Euclidean topology.

A class $\alpha\in N^1(X)_{\mathbb{R}}$, we say that $\alpha$ is nef if and only if $(\alpha \cdot [C])\geq 0$ for any curve $C$.

In this paper, we are often in the case that
$X$ is a projective smooth surface. In this case, we can identify
$N^1(X)_{\mathbb{R}}$ and $N_1(X)_{\mathbb{R}}$. Then we get a natural bilinear form on $N^1(X)_{\mathbb{R}}$.
The following theorem tells us that the signature of this bilinear form is $(1,\dim N^1(X)_{\mathbb{R}}-1)$.

\begin{thm}(Hodge index theorem)\label{Hodge}
Let $X$ be a projective smooth surface, $L,M$ $\mathbb{R}$-divisor,
such that $M\nsim_{num}0$, $(L^2)\geq 0$ and $(L\cdot M)=0$. Then $(M^2)\leq 0$ and the equality holds if and only if $(L^2)= 0$ and
$M\sim_{num}lL$ for some $l\in \mathbb{R}$.
\end{thm}

\subsection{Basics on birational maps on surfaces}
We recall that the resolution of singularities of surface over any algebraically closed field exists (see \cite{Abhyankar1966}).

Let $X,Y$ be projective smooth surfaces. A birational map $f:X\dashrightarrow Y$ is defined by its graph $\Gamma(f)\subseteq X\times Y$, which is an irreducible subvariety for which the projections $\pi_1:\Gamma(f)\rightarrow X$ and $\pi_2:\Gamma(f)\rightarrow Y$ are birational morphisms. Let $I(f)\subseteq X$ denote the finite set of points where $\pi_1$ does not admit a local inverse. We call this set the $indeterminacy$ set of $f$. Let $\mathcal{E}(f)=\pi_1\pi_2^{-1}(I(f^{-1})).$

If $g:Y\dashrightarrow Z$ is another birational map, the graph $\Gamma(g\circ f)$ of the composite map is the closure of the set $$\{(x,g(f(x)))\in X\times Z| x\in X-I(f), f(x)\in Y-I(g)\}.$$ This is equal to the set $$\Gamma(g)\circ\Gamma(f)=\{(x,z)\in X\times Z| (x,y)\in \Gamma(f), (y,z)\in \Gamma(g) \text{ for some $y\in Y$}\},$$ if and only if there is no component $V\subseteq \mathcal{E}(f)$ such that $f(V)\subseteq I(g)$. Here $f(V)$ means $\overline{f(V-I(f))}$.

\begin{thm}(see \cite{Hartshorne1977})Any birational morphism $\pi:X\rightarrow Y$ between smooth surfaces is a composition of finitely many point blowups.

\end{thm}

The following theorem is a corollary of the previous theorem and the resolution of singularities of surface.

\begin{thm}(Factorization theorem)
Any birational map $f: X\dashrightarrow Y$ between smooth surfaces can be written as a composition $f=f_1\circ\cdots \circ f_l\circ \cdots \circ f_n$, where $f_i$ is a point blowup for $i=1,\cdots , l$ and a point blowdown for $i=l+1,\cdots , n$.

\end{thm}

Then we recall the action of a birational map on the N\'eron-Severi group.

Let $f:X\dashrightarrow Y$ be a birational map between smooth surfaces, $\Gamma$ a desingularization of its graph, and denote by $\pi_1: \Gamma\rightarrow X,\pi_2: \Gamma\rightarrow Y$ the natural projections. Then we can define the following linear maps $$f^*=\pi_{1*} \pi_2^*: N^{1}(Y)_{\mathbb{R}}\rightarrow N^{1}(X)_{\mathbb{R}},$$ and $$f_*=\pi_{2*} \pi_1^*: N^{1}(X)_{\mathbb{R}}\rightarrow N^{1}(Y)_{\mathbb{R}}.$$ Observe that $f_*=f^{-1*}$.

\begin{pro}(see \cite{favre})
Let $f:X\dashrightarrow Y$ be a birational map between smooth surfaces.
\begin{points}
\begin{item}
The linear maps $f^*,f_*$ preserve $N^1\subseteq N^1_{\mathbb{R}}$.
\end{item}
\item
If $\alpha\in N^1(Y)_{\mathbb{R}}$ is nef, then $f^*\alpha\in N^1(X)_{\mathbb{R}}$ is nef.
\item
The maps $f^*,f_*$ are adjoint for the intersection form, i.e. $$(f^*\alpha\cdot \beta)=(\alpha\cdot f_*\beta),$$ for any $\alpha\in N^1(Y)_{\mathbb{R}}$ and $\beta\in N^1(X)_{\mathbb{R}}$.

\end{points}

\end{pro}

It is important to note that $f^*$ is not functorial in general. In fact, let $X,Y,Z$ be smooth surfaces, and $f:X\dashrightarrow Y$, $g:Y\dashrightarrow Z$ birational, then for any ample class  $\alpha\in N^1(Z)_{\mathbb{R}}$, $f^*g^*\alpha=(f\circ g)^*\alpha$ if and only if $I(\mathcal{E}(f)\bigcap I(g))=\emptyset.$

Let $\|.\|$ be any norm on $N^1_{\mathbb{R}}$.
\begin{pro}(See \cite{favre})
Let $h:X\dashrightarrow Y$ be a birational map between projective smooth surfaces. Then there is a constant $C>1$ such that for any birational map $f:X\dashrightarrow X$, then $$C^{-1}\|f^*\|\leq \|g^*\| \leq C\|f^*\|,$$ with $g:=h\circ f\circ h^{-1}.$
\end{pro}

\begin{prodefi}(\cite{Dinh2005},\cite{favre})\label{lamda1}
Let $f:X\dashrightarrow X$ be a dominating rational map of a
projective smooth surface, then the first dynamical degree
$$\lambda_1(f):=\lim_{n \rightarrow \infty}\|f^{n*}\|^{1/n}\geq 1$$
exists and is invariant under birational coordinate change and independent of the choice of the norm.

\end{prodefi}

We can defined $\lambda_1$ in another form.
\begin{prodefi}(See \cite{favre})
Let $f:X\dashrightarrow X$ be a dominating rational map of a
projective smooth surface, let $\omega\in N^1(X)_{\mathbb{R}}$ be a big and nef class. We define $\deg_{\omega}(f)=(f^*\omega\cdot \omega)$.

Then $$\lambda_1(f):=\lim_{n \rightarrow \infty}((f^n)^*\omega\cdot \omega)^{1/n}=\lim_{n \rightarrow \infty}(\deg_{\omega}
f^n)^{1/n}.$$

If $L$ is an ample line bundle on $X$, We also write $\deg_L(f)$ for $\deg_{[L]}(f)$.

\end{prodefi}

\begin{prodefi}(see \cite{favre})
Let $f:X\dashrightarrow X$ be a birational map on a smooth surface, for some ample class $\omega\in N^1(X)_{\mathbb{R}}$. Then $f$ is $algebraically$ $stable$ if and only if one of the following holds:
\begin{points}
\item for any $\alpha\in N^1(X)_{\mathbb{R}}$ and any $n\in \mathbb{N}$, one has $(f^*)^n\alpha= (f^n)^*\alpha$;
\item there is no curve $V\subseteq X$ such that $f^n(V)\subseteq I(f)$ for some integer $n\geq 0;$
\item for any $n\geq 0$ one has $(f^n)^*\omega=(f^n)^*\omega .$

\end{points}
\end{prodefi}

In particular, if $X=\mathbb{P}^2$, then $f$ is algebraically stable if and only is $\deg (f^n)=(\deg f)^n$ for any $n\in \mathbb{N}.$

\begin{thm}(\cite{favre})\label{AS modele}
Let $f:X\dashrightarrow X$ be a birational map of a projective
smooth surface, then there is a proper modification $\pi:
\widehat{X}\rightarrow X$ that lifts $f$ to an algebraically stable map.

\end{thm}

\subsection{Classes on the Riemann-Zariski space}
All facts in this subsection can be fined in \cite{boucksomfavrejonsson}. Let $X$ be a projective smooth surface.

For a birational morphism $\pi:X_{\pi}\rightarrow X$, where $X_{\pi}$ is a smooth surface, up to isomorphism, $\pi$ is a finite composition of point blowups. If $\pi$ and $\pi'$ are two  blowups of $X$, we say that $\pi'$ $dominates$ $\pi$ and write $\pi'\geq \pi$ if there exists a birational morphism $\mu:X_{\pi'}\rightarrow X_{\pi}$ such that $\pi'=\pi\circ\mu$. The $Riemann-Zariski\,\, space$ of $X$ is the projective limit $$\mathfrak{X}:=\varprojlim_{\pi} X_{\pi}.$$

\begin{defi} The space of Weil classes of $\mathfrak{X}$ is the projective limit $$W(\mathfrak{X}):=\varprojlim_{\pi}N^1(X_{\pi})_{\mathbb{R}}$$ with respect to the pushforward arrows. The space of Cartier classes on $\mathfrak{X}$ is the inductive limit $$C(\mathfrak{X}):=\varinjlim_{\pi}N^1(X_{\pi})_{\mathbb{R}}$$ with respect to the pullback arrow.
\end{defi}

Concretely, a Weil class $\alpha\in W(\mathfrak{X})$ is given by its $incarnations$ $\alpha_{\pi}\in N^1(X_{\pi})_{\mathbb{R}}$, compatible with pushforwards; that is, $\mu_*\alpha_{\pi'}=\alpha_{\pi}$ as soon as $\pi'=\pi\circ\mu$.

The projection formula  shows that there is an injection $C(\mathfrak{X})\subseteq W(\mathfrak{X})$, so that a Cartier class is a Weil class. In fact, if $\alpha\in N^1(X_{\pi})_{\mathbb{R}}$ is a class in some blowup $X_{\pi}$ of $X$, then $\alpha$ defines a Cartier class, also denoted by $\alpha$, whose incarnation $\alpha_{\pi'}$ in any blowup $\pi'=\pi\circ\mu$ dominating $\pi$ is given by $\alpha_{\pi'}=\mu^*\alpha$. We say that $\alpha$ is $determined$ in $X_{\pi}$. Each Cartier class is obtained in that way.

The spaces $C(\mathfrak{X})$ and $W(\mathfrak{X})$ are birational invariants of $X$. Once the model $X$ is fixed, an alternative and somewhat more explicit description of these spaces can be given in terms of exceptional divisors.

\begin{defi}\label{f1.3}
The set $\mathcal{D}$ of exceptional primes over $X$ is defined as the set of all exceptional prime divisors of all blowups $X_{\pi}\rightarrow X$ modulo the following equivalence relation: two prime divisors $E$ and $E'$ on $X_{\pi}$ and $X_{\pi'}$ are equivalent if the induced birational map $X_{\pi}\dashrightarrow X_{\pi'}$ sends $E$ onto $E'$.
\end{defi}

One can construct an explicit basis for the vector space $C(\mathfrak{X})$ as follows. Let $\alpha_E\in C(\mathfrak{X})$ be the Cartier class determined by the class of $E$ on a model $X_{\pi_E}$. We see that in fact if $E$ and $E'$ are equivalent in the sense in Definition \ref{f1.3} then $\alpha_E=\alpha_{E'}$.
Write $\mathbb{R}^{(\mathcal{D})}$ for the sum $\bigoplus_{\mathcal{D}}\mathbb{R}$.

\begin{pro}
The map $N^1(X)_{\mathbb{R}}\oplus \mathbb{R}^{\mathcal(D)}\rightarrow C(\mathfrak{X})$, sending $\alpha\in N^1(X)_{\mathbb{R}}$ to the Cartier class it determines, and $E\in \mathcal{D}$ to $\alpha_E$ is an isomorphism.
\end{pro}

We now describe $W(\mathfrak{X})$ in terms of exceptional primes. If $\alpha\in W(\mathfrak{X})$ is a given Weil class, let $\alpha_X\in N^1(X)_{\mathbb{R}}$ be its incarnation on $X$. For each $\pi$, the Cartier class $\alpha_{\pi}-\alpha_X$ is determined on $X_{\pi}$ by a unique $\mathbb{R}$-divisor $Z_{\pi}$ exceptional over $X$. If  $E$ is a $\pi$-exceptional prime, we set $\ord_E(\alpha):=\ord_E(Z_{\pi})$ so that $Z_{\pi}=\sum_E \ord_E(Z_{\pi})E.$ It only depends on the class of $E$ in $\mathcal{D}$, by the projection formula. Let $\mathbb{R}^{\mathcal{D}}$ denote the space of all real-valued $\mathbb{R}$ functions on $\mathcal{D}$. Then we obtain a map $W(\mathfrak{X})\rightarrow N^1(X)_{\mathbb{R}}\times \mathbb{R}^{\mathcal{D}}$, which is a homeomorphism (see \cite{boucksomfavrejonsson} a the proof). Then by this homeomorphism, any Weil class $\alpha$ corresponds to a point $(\omega_{\alpha},(\alpha_E)_{E\in \mathcal{D}})\in N^1(X)_{\mathbb{R}}\times \mathbb{R}^{\mathcal{D}}$.

For each $\pi$, the intersection pairing $N^1(X)_{\mathbb{R}}\times N^1(X)_{\mathbb{R}}\rightarrow \mathbb{R}$ is denoted by $(\alpha\cdot \beta )_{X_{\pi}}$. By the pull-back formula, it induces a pairing $W(\mathfrak{X})\times C(\mathfrak{X})\rightarrow \mathbb{R}$ which is denoted by $( \alpha \cdot \beta ).$

\begin{pro}
The intersection pairing induces a topological isomorphism between $W(\mathfrak{X})$ and $C(\mathfrak{X})^*$ endowed with its weak-*topology.

\end{pro}

\begin{defi}  A Weil class $\alpha=(\omega_{\alpha},(\alpha_E)_{E\in \mathcal{D}})$ is said to be in $\mathbb{L}^2$ if and only if $\sum_{E\in \mathcal{D}}a_E^2<+\infty$.

\end{defi}
If $\alpha$ is a Weil class, then $\alpha$ is $\mathbb{L}^2$ if and only if $\inf{(\alpha_X\cdot \alpha_X)}>-\infty$. We denote by $\mathbb{L}^2(\mathfrak{X})$ the space of $\mathbb{L}^2$-classes.
If $\omega\in C(\mathfrak{X})$ is a given class with $(\omega^2)>0$, the intersection form is negative definite on its orthogonal complement $\omega^{\perp}:=\{\alpha\in C(\mathfrak{X})|(\alpha\cdot \omega)=0\}$ as a consequence of the Hodge index theorem applied to each $N^1(X_{\pi})_{\mathbb{R}}$. Then $C(\mathfrak{X})=\mathbb{R}\omega\oplus \omega^{\perp}$ . We let $\mathbb{L}^2(\mathfrak{X}):=\mathbb{R}\omega\oplus \overline{\omega^{\perp}},$ where $\overline{\omega^{\perp}}$ is the completion in the usual sense of $\omega^{\perp}$ endowed the negative definite quadratic form $\alpha\rightarrow(\alpha^2)$. One can define a norm $\|.\|$ on $\mathbb{L}^2(\mathfrak{X})$, such that $\|t\omega\oplus \alpha\|=t^2-(\alpha^2)$, which makes $(\mathbb{L}^2(\mathfrak{X}),\|.\|)$ to a Hilbert space. Though this norm depends on the choice of $\omega$, the topological vector space $\mathbb{L}^2(\mathfrak{X})$ does not depend on.

\begin{pro}
There is a natural continuous injection $\mathbb{L}^2(\mathfrak{X})\rightarrow W(\mathfrak{X})$,and the topology on $\mathbb{L}^2(\mathfrak{X})$ induced by the topology of $W(\mathfrak{X})$ coincides with its weak topology as a Hilbert space.

If $\alpha\in W(\mathfrak{X})$ is a given Weil class, then the intersection number $\alpha_{\pi}$ is a decreasing function of $\pi$, and $\alpha\in \mathbb{L}^2(\mathfrak{X})$ if and only if $(\alpha_{\pi})$ is bounded from below, in which case, $(\alpha^2)=\lim_{\pi}(\alpha_{\pi}^2)$.
\end{pro}

The following theorem is the extension of the Hodge index theorem to the infinite dimensional space $\mathbb{L}^2$.   This fact is crucial in our proof of Theorem \ref{main1}
\begin{pro}\label{hodge index}
If $\alpha, \beta\in \mathbb{L}^2(\mathfrak{X})-\{0\}$ such that
$(\alpha^2)>0$ and $(\alpha\cdot \beta)=0$. Then $(\beta^2)<0$.
\end{pro}

\begin{proof} There are two suits of points $(\alpha_n)_{n\geq 0}$ and $(\beta_n)_{n\geq 0}$ in $C(\mathfrak{X})$ such that
$\alpha_n\rightarrow \alpha$ , $\beta_n\rightarrow \beta$ and $(\alpha_n\cdot \beta_n)=0$ for any $n$. Since
$(\alpha_n^2)\rightarrow(\alpha^2)>0$, we may assume that
$(\alpha_n^2)>0$ for any $n\geq 0$. Then $(\beta^2)=\lim_{n\rightarrow\infty}(\beta_n^2)\leq 0$ by Hodge index theorem.

We only have to show that $(\beta^2)\neq 0$. Otherwise
$(\beta^2)=0$. Since $\beta\neq 0$, there is a $\gamma\in
\mathbb{L}^2(\mathfrak{X})$ such that $(\beta\cdot\gamma)>0$. Then for $0<t\ll 1$,
$$(\beta+t\gamma)^2=2(\beta\cdot\gamma)t+(\gamma^2)t^2>0.$$ And
since $0<t\ll 1$, $$(\beta\cdot\gamma)-(\gamma^2)>0, \text{ and }
((\alpha-\frac{(\alpha\cdot\gamma)}{(\beta\cdot\gamma)-(\gamma^2)}t)^2)>0.$$
It is easy to check that
$$((\alpha-\frac{(\alpha\cdot\gamma)}{(\beta\cdot\gamma)-(\gamma^2)}t)\cdot
(\beta+t\gamma))=0.$$ Thus $((\beta+t\gamma)^2)\leq 0$, by the proof
at the previous paragraph. It contradicts the fact that
$(\beta+t\gamma)^2>0$.
\end{proof}
Let $f:X\dashrightarrow Y$ be a dominant rational map between two smooth surfaces. For each blowup $Y_{\varpi}$ of $Y$, there is a blowup $X_{\pi}$ of $X$ such that the induced map $X_{\pi}\rightarrow Y_{\varpi}$ is regular. The associated pushforward map $N^1(X_{\pi})_{\mathbb{R}}\rightarrow N^1(Y_{\varpi})_{\mathbb{R}}$ and pullback map $N^1(Y_{\varpi})_{\mathbb{R}}\rightarrow N^1(X_{\pi})_{\mathbb{R}}$ are compatible with the projective and injective systems defined by pushforwards and pullbacks that define Weil and Cartier class.

\begin{defi}
Let $f:X\dashrightarrow Y$ be a dominant rational map between two smooth surfaces. Let $\mathfrak{Y}$ be the Riemann-Zariski space of $Y$. We denote by $f_*: W(\mathfrak{X})\rightarrow W(\mathfrak{Y})$ the induced pushforward operator and by $f^*:C(\mathfrak{Y})\rightarrow C(\mathfrak{X})$ the induced pullback operator.
\end{defi}

\begin{pro}
The pullback $f^*:C(\mathfrak{Y})\rightarrow C(\mathfrak{X})$ extends to a continuous operator $f^*:\mathbb{L}^2(\mathfrak{Y})\rightarrow \mathbb{L}^2(\mathfrak{X})$, such that $$((f^*\alpha)^2)=(\alpha^2)$$ for any $\alpha\in \mathbb{L}^2(\mathfrak{X})$. By duality, the pushforward map $f_*:W(\mathfrak{Y})\rightarrow W(\mathfrak{X})$ induces a continuous operator $f_*:\mathbb{L}^2(\mathfrak{Y})\rightarrow \mathbb{L}^2(\mathfrak{X})$, so that $$((f_*\alpha)^2)=(\alpha^2)$$ for any $\alpha\in \mathbb{L}^2(\mathfrak{X})$. Moreover, we have $$(f^*\alpha\cdot \beta)=(\alpha\cdot f^*\beta)$$ for any $\alpha,\beta\in \mathbb{L}^2(\mathfrak{X})$ and $f_*=f^{-1*}$.
\end{pro}

\begin{defi} let $L$ be an ample bundle on $X$ and $\mathcal{L}=[L]\in \mathbb{H}(\mathfrak{X})$, then
we define $\mathbb{H}(\mathfrak{X})$ to be the subset $\{\alpha\in \mathbb{L}^2(\mathfrak{X})| \alpha^2=1 \text{ and } (\alpha\cdot \mathcal{L})>0\}$.

\end{defi}

\subsection{Hyperbolic spaces }
In this section, we will review some properties of hyperbolic spaces in the sense of Gromov. In fact the space $\mathbb{H}(\mathfrak{X})$ is  Gromov hyperbolic. This fact is the basis for the work of Cantat in \cite{Cantat}, and it is very important in the proof of Theorem \ref{f^2>2^18f}.

Recalled that a metric space $(M,d)$ is geodesic if and only if for any two point $x,y\in X$, there are at least one geodesic line jointing them.

\begin{defi}
For a given number $\delta\geq0$, a metric space $(M,d)$
satisfies the condition of Rips of constant $\delta$ if it is geodesic and for any
geodesic triangle $$\Delta=\bigcup[x,y]\bigcup[y,z]\bigcup[z,x]$$ of
$X$ and any $$u\in [y,z],$$ we have $$d(u,[x,y]\bigcup[z,x])\leq
\delta.$$ A space $X$ is called $hyperbolic\,\, in\,\, the\,\,
sense\,\, of\,\, Gromov$ if there is a number $\delta\geq0$ such
that $M$ satisfies the Rips  condition of constant $\delta$.

\end{defi}

\lem\label{H^2}(\cite{Coornaert1990}) The hyperbolic plane $\mathbb{H}^2$ satisfies the
Rips condition of $\log3$.
\endlem

For any $\alpha,\beta\in \mathbb{H}(\mathfrak{X})$, we define
$$d_{\mathbb{H}(\mathfrak{X})}(\alpha,\beta)=(\cosh)^{-1}(\alpha,\beta),$$
where
$\cosh(x)=(e^x+e^{-x})/2$ as in \cite{Cantat}.
\lem
The function $d_{\mathbb{H}(\mathfrak{X})}$ defines a distance on the space $\mathbb{H}(\mathfrak{X})$ and satisfies
the condition of Rips of $\log3$.
\endlem
\proof  For any $\alpha,\beta,\gamma \in
\mathbb{H}(\mathfrak{X})$. Let $V$ be the $3$ dimensional subspace
of $\mathbb{L}^2(\mathfrak{X})$. Since $(\alpha\cdot\alpha)=1$, for
any $\xi\in V-\{0\}$ such that $(\alpha\cdot\xi)=0$, we have that
$(\xi^2)<0$ by Proposition \ref{hodge index}. So $q_{|V}$ is a
quadratic form of type (1,-1,-1). Hence
$H=\mathbb{H}(\mathfrak{X})\bigcap V$ with the distance function
$d_V(\eta,\xi)=(\cosh)^{-1}(\eta\cdot \xi)_{V}=(\cosh)^{-1}(\eta\cdot \xi)=d_{\mathbb{H}(\mathfrak{X})}(\eta\cdot\xi)$
is a hyperbolic plane. Then $(V,d_V)$ satisfies the condition of
Rips of $\log3$. So $d_{\mathbb{H}(\mathfrak{X})}$ is a distance
function and
$(\mathbb{H}(\mathfrak{X}),d_{\mathbb{H}(\mathfrak{X})})$ satisfies
the condition of Rips of $\log3$ since $d_V$ does for any
$\alpha,\beta,\gamma\in\mathbb{H}(\mathfrak{X})$.
\endproof

\begin{thm}(Theorem 5.16 of \cite{harpe})\label{5.16}
Let $(M,d)$ be a separable (i.e. having a countable dense subset) geodesic and hyperbolic metric space
which satisfies the condition of Rips of constant $\delta$. If
$(x_i)_{0\leq i \leq n}$ is a sequence of points such that
$$d(x_{i+1},x_{i-1})\geq
\max(d(x_{i+1},x_{i}),d(x_{i},x_{i-1}))+18\delta+\kappa$$ for some
constant $\kappa>0$ and $i=1,\cdots,n-1$. Then
$$d(x_n,x_0)\geq \kappa n.$$

\end{thm}

\section{Effective bounds on $\lambda_1$}\label{deg}

\begin{thm}\label{f^2>2^18f}(=Theorem \ref{f^2>2^18f1})
Let $X$ be a projective surface over $\mathbf{k}$, $L$ be an ample line bundle on $X$ and $f:X\dashrightarrow X$ a birational map.
If $q=\frac{\deg_{L}(f^2)}{3^{18}\sqrt{2}\deg_{L}(f)}>1 $, then we have $$\lambda_1(f)>\frac{(4\times3^{36}-1)q^2+1}{4\times3^{36}q}\geq1.$$

\end{thm}

\begin{proof}

Let $\mathcal {L}_n=f^{*n}\mathcal {L}\in \mathbb{H}(\mathfrak{X})$
for $n>0$. Since $f^*$ preserves the intersection form
of $\mathbb{L}^2(\mathfrak{X})$, it is an isometry of
$\mathbb{H}(\mathfrak{X})$. Then
$$d_{\mathbb{H}(\mathfrak{X})}(\mathcal {L}_{n+1},\mathcal
{L}_{n-1})=d_{\mathbb{H}(\mathfrak{X})}(\mathcal {L}_{2},\mathcal
{L})=\cosh^{-1}(\deg_{L}(f^2))=\cosh^{-1}(3^{18}\sqrt{2}\deg_{L}(f)q)$$
for any $n\geq1$. We claim that for any $u,q\geq1$,
$$\cosh^{-1}(3^{18}\sqrt{2}uq)>\cosh^{-1}(u)+18\log3+\log\Big(\frac{(4\times3^{36}-1)q^2+1}{4\times3^{36}q}\Big).$$

Then
$$\cosh^{-1}(3^{18}\sqrt{2}\deg_{L}(f)q)>\cosh^{-1}(\deg_{L}(f))+18\log3+\log\Big(\frac{(4\times3^{36}-1)q^2+1}{4\times3^{36}q}\Big).$$ Let
$\kappa>\log\Big(\frac{(4\times3^{36}-1)q^2+1}{4\times3^{36}q}\Big)>0$ such that $$\cosh^{-1}(3^{18}\sqrt{2}\deg_{L}(f)q)>\cosh^{-1}(\deg_{L}(f))+18\log3+\kappa.$$
Then $$d_{\mathbb{H}(\mathfrak{X})}(\mathcal {L}_{n+1},\mathcal
{L}_{n-1})>\cosh^{-1}(\deg_{L}(f))+18\log3+\kappa.$$ Since
$d_{\mathbb{H}(\mathfrak{X})}(\mathcal {L}_{n+1},\mathcal
{L}_{n})=\cosh^{-1}(\deg_{L}(f))$, we obtain
$$d_{\mathbb{H}(\mathfrak{X})}(\mathcal {L}_{n+1},\mathcal
{L}_{n})>\max(d_{\mathbb{H}(\mathfrak{X})}(\mathcal {L}_{n+1},\mathcal
{L}_{n}),d_{\mathbb{H}(\mathfrak{X})}(\mathcal {L}_{n},\mathcal
{L}_{n-1}))+18\log3+\kappa.$$

Let $W$ be the subspace of $\mathbb{H}(\mathfrak{X})$ spanned by
$\{\mathcal {L}_n\}$. Then $W$ is separable and for any $x,y\in W$,
$[x,y]\subseteq W$. Let
$d_W(x,y)=d_{\mathbb{H}(\mathfrak{X})}(x,y)$. Then $(W,d_W)$ is a
separated geodesic and hyperbolic metric space which satisfies the
Rips condition of constant $\log3$.

Thus by Theorem \ref{5.16}, we get
$$\cosh^{-1}(\deg_{L}(f^n))=d_{\mathbb{H}(\mathfrak{P})}(\mathcal {L}_{n},\mathcal
{L})>\kappa n,$$ which is equivalent to $$\deg_{L}(f^n)>(e^{\kappa n}+e^{-\kappa
n})/2>e^{\kappa n}/2,$$ and we conclude that $\lambda_1(f)\geq e^{\kappa}>\frac{(4\times3^36-1)q^2+1}{4\times3^36q}$.

Let us prove the claim above.  For any $u\geq 1$, $q>1$,
$$\cosh^{-1}(3^{18}\sqrt{2}u)=\log(3^{18}\sqrt{2}u+\sqrt{2\times3^{36}u^2+1})>\log(3^{18}u+\sqrt{2\times3^{36}u^2})$$$$=18\log3+\log(u+\sqrt{2u^2})\geq18\log3+\log(u+\sqrt{u^2+1})=\cosh^{-1}(u)+18\log3.$$

And $$\cosh^{-1}(3^{18}\sqrt{2}uq)-\cosh^{-1}(3^{18}\sqrt{2}u)=\log\Big(\frac{3^{18}\sqrt{2}uq+\sqrt{2\times3^{36}u^2q^2+1}}{3^{18}\sqrt{2}u+\sqrt{2\times3^{36}u^2+1}}\Big).$$
It follows that $$\frac{3^{18}\sqrt{2}uq+\sqrt{2\times3^{36}u^2q^2+1}}{3^{18}\sqrt{2}u+\sqrt{2\times3^{36}u^2+1}}=q-\frac{\sqrt{2\times3^{36}u^2q^2+q^2}-\sqrt{2\times3^{36}u^2q^2+1}}{3^{18}\sqrt{2}u+\sqrt{2\times3^{36}u^2+1}}$$$$=q-\frac{q^2-1}{(3^{18}\sqrt{2}u+\sqrt{2\times3^{36}u^2+1})(\sqrt{2\times3^{36}u^2q^2+q^2}+\sqrt{2\times3^{36}u^2q^2+1})}$$$$\geq q-\frac{q^2-1}{4\times3^{36}q}=\frac{(4\times3^{36}-1)q^2+1}{4\times3^{36}q},$$ which concludes the proof.
\end{proof}
By Theorem \ref{f^2>2^18f}, we can estimate $\lambda_1$, using $\deg_{L}(f^n)$ and
$\deg_{L}(f^{2n})$ for $n\gg0$.

\begin{cor}\label{kappa} Let $f$ be a birational map of a projective surface $X$ over $\mathbf{k}$.
For any integer $n>0$, let $q_n=\frac{\deg_{L}(f^{2n})}{3^{18}\sqrt{2}\deg(f^n)}$. If $q_n\geq 1$, then $$(q_n/2)^{1/n}<\lambda_1(f).$$ And $$\lim_{n\rightarrow\infty}(q_n/2)^{1/n}=\lambda_1(f).$$
\end{cor}

\proof
If $$q_n=\frac{\deg_{L}(f^{2n})}{3^{18}\sqrt{2}\deg_{L}(f^n)}\geq1.$$ We see that
$$\lambda_1(f)^n=\lambda_1(f^n)>
\frac{(4\times3^{36}-1)q_n^2+1}{4\times3^{36}q_n}$$ by Theorem \ref{f^2>2^18f}.
Then
$$\lambda_1(f)\geq \Big(\frac{(4\times3^{36}-1)q_n^2+1}{4\times3^{36}q_n}\Big)^{1/n}>\Big(\frac{(4\times3^{36}-2)}{4\times3^{36}}q_n\Big)^{1/n}>(q_n/2)^{1/n}.$$

To conclusion, we see that
$$\lim_{n\rightarrow\infty}(q_n/2)^{1/n}=\lim_{n\rightarrow
\infty}\Big(\frac{\deg_{L} (f^{2n})}{\deg_{L}
(f^{n})}\Big)^{1/n}/ \lim_{n\rightarrow
\infty} (2\times3^{18}\sqrt{2})^{1/n}$$$$=(\lim_{n\rightarrow \infty}(\deg_{L}
(f^{n}))^{1/n})^2/\lim_{n\rightarrow \infty}(\deg_{L}
(f^{n}))^{1/n}=\lambda_1(f)^2/\lambda_1(f)=\lambda_1(f).$$
\endproof

\section{The behavior of $\lambda_1$ in family}\label{variation}
In this section, we apply Theorem \ref{f^2>2^18f} and Corollary \ref{kappa} to study the behavior of the first dynamical degree in families of birational surface maps.

\subsection{Lower semi-continuity of $\lambda_1$}\label{sslsc}
We first prove a version of Theorem 1.3 in the general context of integral schemes. We should rely on the following lemma.
\begin{lem}\label{lowersemicontinuous}
Let $S$ be a smooth integral scheme, $\pi: X\rightarrow S$ be a projective smooth and surjective morphism, $L$ be a relatively nef line bundle on $X$ and $f:  X\dashrightarrow X$ be a birational map over $S$. Let $U$ be the  sub-open set of $S$ such that for any point $p\in U$, $f$ induces a birational map $f_p$ of the special fiber $X_p$, and write $L_p=L_{|X_p}$. Then $p\rightarrow \deg_{L_p}(f_p)$ is a lower semi-continuous function on $U$.

\end{lem}

\proof
 Let $\Gamma\subseteq X\times_SX$ be the graph of $f$, $\pi_1$,$\pi_2$ the natural projections onto $X$ such that $\pi_2\circ \pi_1^{-1}=f$. For any point $x\in U$, let $\Gamma_x$ be the fiber of $\Gamma$ above $x$, and $\pi_{1x},\pi_{2x}$ the restriction of $\pi_1,\pi_2$ on $\Gamma_x$.  Let $\kappa$ be the generic point of $S$. Then $$\int_{\Gamma_{x}}\pi_{1x}^*L_x\cdot \pi_{2x}^*L_x=\int_{\Gamma_{\kappa}}\pi_{1\kappa}^*L_{\kappa}\cdot \pi_{2\kappa}^*L_{\kappa}=\deg_{L_\kappa}(f_{\kappa})$$ is a constant function on $U$ \cite{Fulton1984}. For any $x\in U$, $\Gamma_x$ may have many irreducible components, but there is only one component $\Gamma_x'$ satisfies $\pi_{1x}(\Gamma_x')=X$. Then $\deg_{L_x}(f_{x})=\int_{\Gamma_x'}\pi_{1x}L_x\cdot \pi_{2x}L_x\leq\int_{\Gamma_{x}}\pi_{1x}L_x\cdot \pi_{2x}L_x.$ And there is an nonempty open set $V$ of $U$ such that for any point $x\in V$, $\Gamma_x$ is irreducible, then $\deg_{L_x}(f_{x})=\int_{\Gamma_{x}}\pi_{1x}^*L_x\cdot \pi_{2x}^*L_x=\deg_{L_\kappa}(f_{\kappa})$.

 For any $\lambda\in \mathbb{R}$, define $R=\{x\in U| \deg_{L_x}(f_{x})\leq \lambda\},$ and set $Z=\overline{R}$. Supose there is a point $x\in Z-R$, then we can find an irreducible component $Y$ of $Z$ such that $x\in Y$. Then $\overline{Y\bigcap Z}=Y$, let $W=Y\bigcap Z$ and $\eta$ the generic point of $Y$. By the previous argument, we see that there is an open $O$ of $Y$, such that for any point $y\in O$, $\deg_{L_y}(f_{y})=\deg_{L_\eta}(f_{\eta})$. We may pick $y\in O\bigcap W$. Then $\lambda<\deg_{L_x}(f_{x})\leq \deg_{L_\eta}(f_{\eta})=\deg_{L_y}(f_{y})\leq \lambda.$ We get a contradiction. So $R=Z$ is closed. Then we conclude that $\deg_{L_x}(f_{x})$ is lower semi-continuous on $U.$

\endproof

\begin{thm}\label{generally}
 Let $S$ be an integral scheme, $\pi: X\rightarrow S$ a projective smooth and surjective morphism and $f:  X\dashrightarrow X$ be a birational map over $S$. Let $U$ be the open set of points $p\in S$ for which the reduction $f_p$ is a birational map. Then the function $F: p\in U\rightarrow \lambda_1(f_x)$ is lower semi-continuous.

\end{thm}

\begin{proof}Denote by $\kappa$ the generic point of $S$ and $\lambda_{1,\kappa}$ the first dynamical degree of $f_{\kappa}$.
 As in the proof of Lemma \ref{lowersemicontinuous}, we only have to show that for any $\lambda<\lambda_{1,\kappa}$,  there is a nonempty open set $U_{\lambda}$ of $U$, such that for any point $p\in U_{\lambda}$, $\lambda_{1,\kappa}\geq \lambda_1(f_p)\geq \lambda$.

For any $p\in U, n>0$, $ \deg_{L_p}(f^n_p)\leq\deg_{L_{\kappa}}(f^n_{\kappa})$ hence
$$\lambda_1(f_p)\leq \lambda_1(f_{\kappa}).$$

The theorem trivially holds in the case $\lambda\leq 1$, so we may assume that $\lambda_{1,\kappa}>\lambda>1$.
For any
$\lambda_{1,\kappa}>\lambda>1$, there is an integer $n$ such that
$$\Big(\frac{\deg_{L_{\kappa}}(f_{\kappa}^{2n})}{2\times 3^{18}\sqrt{2}\deg_{L_{\kappa}}(f_{\kappa}^n)}\Big)^{1/n}>\lambda >1$$ by Corollary \ref{kappa}.
By Lemma \ref{lowersemicontinuous}, there is an nonempty open set $U_{\lambda}\subseteq U$ such that for any $p\in U_{\lambda}$,
 $\deg_{L_p}(f_p^n)=\deg_{L_{\kappa}}(f_{\kappa}^n)$ and
$\deg_{L_p}(f_p^{2n})=\deg_{L_{\kappa}}(f_{\kappa}^{2n})$. We conclude that
$$\Big(\frac{\deg_{L_p}(f_{p}^{2n})}{2\times 3^{18}\sqrt{2}\deg_{L_p}(f_{p}^n)}\Big)^{1/n}>\lambda,$$
and by Corollary \ref{kappa}, for any $p\in U_{\lambda}$, $$\lambda_1(f_p)\geq
\Big(\frac{\deg_{L_p}(f_{p}^{2n})}{2\times 3^{18}\sqrt{2}\deg_{L_p}(f_{p}^n)}\Big)^{1/n}>\lambda,$$ as required.
\end{proof}

\begin{cor} Let $S$ be an integral scheme, $\pi: X\rightarrow S$ be a projective smooth and surjective morphism, $f:  X\dashrightarrow X$ be a birational map over $S$ such that for any $p\in S$, $f_p$ is birational. Then there is an integer $M>0$ such that for any $p\in S$, $\lambda_1(f_p)=1$ if and only if $$\frac{\deg_{L_p}(f_{p}^{2n})}{2\times 3^{18}\sqrt{2}\deg_{L_p}(f_{p}^n)}<1$$ for any $n=1,2,\cdots,M.$

\end{cor}

\proof
 Fix any integer $m>0$, we define $Z_m=\{p\in S|\frac{\deg_{L_p}(f_{p}^{2n})}{2\times 3^{18}\sqrt{2}\deg_{L_p}(f_{p}^n)}<1, \text{ for any } 0<n\leq m\}$ and $Z=\{p\in S| \lambda_1(f_p)=1\}$. We observe that $Z_m\supseteq Z_{m+1}$ for any $m\geq 1$, and that by Theorem \ref{generally}, $Z$ is closed. By Corollary \ref{kappa}, we have $Z=\bigcap_{m\geq 1} Z_m$. Let $Y_m=\overline{Z_m}$, then $Y_{m}\supseteq Y_{m+1}$ for any $m\geq 1$. Since $Y_m$ is closed, there is an integer $M>0$ such that $Y_M=\bigcap_{m\geq 1}Y_m$. Then $Z\subseteq Y_M.$

 If $Z\neq Y_M$, there is a point $x\in Y_M-Z$. Let $Y$ be an irreducible component of $Y_M$ containing $x$ and $\eta$ the generic point of $Y$. Then $Y=\overline{Y\bigcap Z_N}$ for any $N\geq M$. Since $\lambda_1(f_x)>1$, we have $\lambda_1(f_{\eta})>1$ by Lemma \ref{lowersemicontinuous}. There exists $N\geq M$ such that $\eta$ is not in $Z_N$. Then $\frac{\deg_{L_{\eta}}(f_{\eta}^{2N})}{2\times 3^{18}\sqrt{2}\deg_{L_{\eta}}(f_{\eta}^N)}\geq 1$. By lemma \ref{lowersemicontinuous}, there is a sub-open set $U$ of $Y$ such that for any point $y\in U$ ,$n=1,2,\cdots , N$,  we have $\deg_{L_y}(f_{y}^n)=\deg_{L_{\eta}}(f_{\eta}^n)$ and $\deg_{L_y}(f_{y}^{2n})=\deg_{L_{\eta}}(f_{\eta}^{2n})$. Then $U\bigcap Z_N=\emptyset$. It contradicts the fact that $Y=\overline{Y\bigcap Z_N}$. So we get $Z= Y_M$. Then $Y_M\supseteq Z_M\supseteq Z= Y_M$, so $Z=Z_M$.

\endproof

\subsection{Proof of Theorem \ref{1general}}\label{pot1g}
 We denote by $\Bir_d$ the space of birational maps of $\mathbb{P}^2(\mathbf{k})$ with degree $d$. It has a natural algebraic structure which make it to be a quasi-projective  space.

By Theorem \ref{generally}, it is easy to see that if a component of $\Bir_d$ contains a point with $\lambda_1>1$, then $\lambda_1>1$ for a general point in this component. However since there are no good description of the components of $\Bir_d$ for $d\geq 3$ \cite{bir}, it is not a priori obvious to see that if any component contains a such point.

\begin{thm}(=Theorem \ref{1general})\label{general} Let $d\geq 2$ be an integer and $\Bir_d$ the space of birational maps of $\mathbb{P}^2(\mathbf{k})$ of degree $d$. Then for any $\lambda<d$, $U_{\lambda}=\{f\in \Bir_d | \lambda_1(f)>\lambda\}$ is a Zariski dense open set of $\Bir_d$.

In particular, for a general birational map $f$ of degree $d>1$, we have $\lambda_1(f)>1$.
\end{thm}

\begin{rem}If the base field is uncountable, then the set $$\{f\in \Bir_d| \lambda_1(f)=d\}\supseteq \bigcap_{n=1}^{\infty}U_{d-1/n}$$ is dense in $\Bir_d$. So for any very general point $f\in \Bir_d$ we have $\lambda_1(f)=d$.

\end{rem}

\begin{rem}In fact, we shall see that in any $\PGL(3)$ orbit, points with $\lambda_1>\lambda$ are dense.

\end{rem}

\proof[Proof of Theorem \ref{general}]
We claim that for any irreducible component $S$ of $\Bir_d$, there is a point $f\in S$ such that $\lambda_1(f)>\lambda$.

\proof[Proof of the claim]

Choose $f\in\Bir_d$, consider the map $$T_f:\PGL(3)\rightarrow \Bir_d$$ sending $A$ to $A\circ f$. Let $I(f)=\{x_1,x_2,\dots,x_n\}$, and $I(f^{-1})=\{z_1,z_2,\cdots,z_m\}$ be the indeterminacy sets of $f$ and $f^{-1}$ respectively. For any $i=1,2,\cdots,n$, there is a curve $C_i$ such that $f(C_i)=x_i$. Let $y_i$ be a point in $C_i-(I(f)\bigcup I(f^{-1}))$. We can find a point $A_i\in \PGL(3)$ such that $A_i(x_i)=y_i$. Then for any $n\geq 0$ $(A_i\circ f)^n\circ A_i(x_i)=y_i$.

For any $i=1,2,\cdots,n$, we define the map $$V_{1,i}: \PGL(3)\rightarrow \mathbb{P}^2$$ by $V_{1,i}(A):=A(x_i)$. Let $U_{0,i}=\PGL(3)$ and set $$U_{1,i}=V_{1,i}^{-1}(\mathbb{P}^2-I(f)).$$ Then $U_{1,i}$ is an open set of $\PGL(3)$. Since $A_i(x_i)=y_i$ is not in $I(f)$, $A_i\in U_{1,i}$. So $U_{1,i}$ is not empty. Now we define the map $$V_{2,i}: U_{1,i}\rightarrow \mathbb{P}^2,$$ by $$V_{2,i}(A)=A\circ f\circ A(x_i)$$  and set $$U_{2,i}=V_{2,i}^{-1}(\mathbb{P}^2-I(f)).$$  Since $A\circ f\circ A_i(x_i)=y_i$, $U_{2,i}$ is as before an open set containing $A_i$.

By induction, for any $i$ we build a sequence of non empty open subsets of $\PGL(3)$ $A_i\in U_{l,i}\subseteq U_{l-1,i}$ and maps $V_{l+1,i}: U_{l,i}\rightarrow \mathbb{P}^2$ sending $A$ to $(A\circ f)^{l}\circ A(x_i)$. For any $A\in U_{l,i}$, $(A\circ f)^{t}\circ A(x_i)$ is not in $I(f)$ for $t=0,\cdots, l-1$. Let $U_l=\bigcap_{i=1}^nU_{l,i}$. Since the $U_{l,i}$'s are nonempty and $\PGL(3)$ is irreducible, then $U_l$ is also non empty and Zariski dense. For any $A$ in $U_l$, we have $\deg((A\circ f)^{s})=(\deg(A\circ f))^{s}=d^s$ for $s=0,1,\cdots, l+1.$

Pick $l$ sufficiently large such that $$\frac{(4\times3^{36}-1)(d^l/3^{18}\sqrt{2})^2+1}{4\times3^{36}(d^l/3^{18}\sqrt{2})}>\lambda.$$ For any  point $A\in U_{2l-1}$, $\deg (A\circ f)^{2l}=d^{2l}$ and $\deg (A\circ f)^{l}=d^{l}$, hence by Theorem \ref{f^2>2^18f}, $\lambda_1(A\circ f)>\lambda$.

Pick any irreducible component $S$ of $\Bir_d$. There is a birational map $$G:\mathbb{P}^2\times S\dashrightarrow \mathbb{P}^2\times S$$ over $S$ by $G(x,f)=(f(x),x).$  Then the map $G_f$ on the fiber at $f\in S$ induced by $G$ is exactly $f$.
For any $f\in S-$\{points in other components\}, we have $A\circ f \in S$ for any $A\in \PGL(3)$ since $\PGL(3)$ is irreducible. By the discussion of the previous paragraph, there is a point $A\in \PGL(3)$ such that $\lambda_1(A\circ f)>\lambda$. Let $\kappa$ be the generic point of $S$, then by Theorem \ref{generally}, $\lambda_1(G_{\kappa})\geq\lambda_1(G_{A\circ f})= \lambda_1(A\circ f)>\lambda$. Again by Theorem \ref{generally}, for a general point $f$ in $S$, we have $\lambda_1(f)>\lambda$.
\endproof

 By Theorem \ref{generally}, we see that $F$ is lower semi-continue on each component of $\Bir_d$. So itself is lower semi-continue. Then $U_{\lambda}=F^{-1}((\lambda,+\infty])$ is open. Since $U_{\lambda}$ meets all the irreducible components, it is dense.
\endproof

\section{An example}\label{example}
In this section we provide an example of a birational map $f$ over $\mathbb{Z}$ such that $\lambda_1(f_p)<\lambda_1(f)$ for any prime number $p>0$.
We introduce two birational maps $g=[xy,xy+yz,z^2]$ and $h=[x,x-2z,-x+y+3z]$.

\begin{pro}\label{fAS}The map $f=h\circ g=[xy,xy-2z^2,yz+3z^2]$ is algebraically stable.
\end{pro}
\proof We see that $f^{-1}=[2x^2-2xy,(-3x+3y+2z)^2,(x-y)(-3x+3y+2z)]$, $I(f)=\{[1,0,0],[0,1,0]\}$ and $I(f^{-1})=\{[1,1,0],[0,-2,3]\}$. Let $C$ be the curve defined by the equation $x=0$. Then $C$ is $f$-invariant. We see that $f$ induces an automorphism $f_{|C}$ on $C$ sending $[0,y,z]$ to $[0,-2z,y+3z]$. We compute the orbits of the points in $I(f^{-1})$. Since $[1,1,0]$ is a fixed point of $f$. The orbit of $[1,1,0]$ is $\{[1,1,0]\}$ which does not meet $I(f^{-1})$. Since $C$ is $f$-invariant, the orbit of $[0,-2,3]$ is contained in $C$. Let $i$ be an automorphism of $C$ sending $[0,y,z]$ to $[0,y-2z,-y+z]$, then $i^{-1}\circ f_{|C}\circ i$ sends $[0,y,z]$ to $[0,2y,z]$, $i^{-1}([0,1,0])=[0,1,1]$ and $i^{-1}([0,-2,3])=[0,4,1]$. So the orbit of $[0,-2,3]$ is $i(\{[0,2^{l+2},1]|l=0,1,2\cdots \})$ which does not meet $I(f)$. Then $f$ is algebraically stable.
\endproof

Since $f$ is algebraically stable, $\lambda_1(f)=\deg(f)=2$. The following proposition shows that for any prime $p>2$, $f_p$ is a birational map of $\mathbb{P}^2(\overline{\mathbb{F}_p})$ and $\lambda_1(f)<2$, observe that $f_p$ is not dominant.

\begin{pro}
For any prime $p>2$, $f_p$ is a birational map of $\mathbb{P}^2(\overline{\mathbb{F}_p})$ and $\lambda_1(f)<2$.
\end{pro}

\begin{proof}
Recall that $f_p=[xy,xy-2z^2,yz+3z^2]$. Let $j=[2x^2-2xy,(-3x+3y+2z)^2,(x-y)(-3x+3y+2z)]$. Then $j\circ f_p=[4xyz^2,4y^2z^2,4yz^3]=[x,y,z]$. So $f_p$ is birational, $f_p^{-1}=j$ and $\deg(f)=2$. Thus $I(f_p)=\{[1,0,0],[0,1,0]\}$ and $I(f_p^{-1})=\{[1,1,0],[0,-2,3]\}$. Let $C$ be the curve defined by the equation $x=0$. Then $C$ is $f_p$-invariant. Since $C$ is $f_p$-invariant, the orbit of $[0,-2,3]\}$ is contained in $C$. Let $i$ be an automorphism of $C$ sending $[0,y,z]$ to $[0,y-2z,-y+z]$, then $l=i^{-1}\circ f_{|C}\circ i$ sends $[0,y,z]$ to $[0,2y,z]$, $i^{-1}([0,1,0])=[0,1,1]$ and $i^{-1}([0,-2,3])=[0,4,1]$. Since $p>2$, we have $2^{p-1}=1$. Let $n_p\geq 2$ be the order of $2$ in the multiplicative group $\mathbb{F}_p^{\times}$. So $l^{n_p-2}([0,4,1])=[0,1,1]$. So $f^{n_p-2}([0,-2,3])=[0,1,0]$$\in I(f_p)$, and $f_p$ is not algebraically stable. Then there is a number $n$ such that $\deg(f_p^n)<\deg(f_p)^n=2^n$ (in fact the least number having this property is $n_p-1$), and $\lambda_1(f_p)\leq \deg(f_p^n)^{1/n}<2$ by Corollary \ref{kappa}.
\end{proof}

In fact, we can
compute $\lambda_1(f_p)$ by constructing an
algebraically stable model. For $p>2$, $\lambda_1(f)$ is the greatest real root of the polynomial $$x_p^{n_p}-2x_p^{n_p-1}+1=0.$$ Since $(x_p-2)x_p^{n_p-1}+1=0$, then $x_p<2$. Define $F_n(x)=(x-2)x^{n-1}+1$. When $n>2$, $F_n(3/2)<0,$ and $F_n(2)=1>0.$ So the
largest root $x$ of $F_n(x)=0$ satisfies $2>x> 3/2.$ Then $2^{n_p-1}(x_p-2)+1<0=(x_p-2)x_p^{n_p-1}+1<(x_p-2)(3/2)^{n_p-1}+1,$ and we get
$2-(1/2)^{n_p-1}>x_p>2-(2/3)^{n_p-1}.$

\section{The case $\lambda_1>1$}\label{case1}

The purpose of this section is to prove Theorem \ref{main1}.

\subsection{The case of finite fields}
At first, we recall the following theorem of Hrushovski.

\begin{thm}(\cite{hu7})\label{hrushovski}
Let $V$ be an irreducible affine variety of degree $d$, dimension $r$
over an algebraically closed field of characteristic $p$, $q$ a
power of $p$, $S$ $\subseteq V\times V$ an irreducible subvariety of
degree $\delta$ of dimension $r$ such that both projections are
dominant and the second one is quasi-finite. Let $\Phi_q \subseteq
V\times V$ be the graph of the $q$-Frobenius map. Set
$u=\frac{\deg\pi_1}{\deg_{ins}\pi_2}$.Then there is a constant
$C(d,\delta,r)$ that only dependents on $d,\delta,r$, such that if
$q>C(d,\delta,r)$, then
$$|\#(S \bigcap \Phi_q )-uq^r|\leq C(d,\delta,r)q^{r-1/2}.$$

\end{thm}

Building on [Proposition 5.5, \cite{fa}], we show that the periodic points of birational maps are Zariski dense over a finite field. In \cite{fa}, $\phi$ is assumed to be regular, but in our case $\phi$ is just birational. So the proofs are little different. We denote by $\mathbb{F}$ the algebraic closure of a
finite field.

\begin{pro}\label{Fakhruddin 5.5}
Let $X$ be an algebraic variety over $\mathbb{F}$ and $\phi:X\dashrightarrow X$ be a birational map. Then the
subset of $X(\mathbb{F})$ consisting of non-critical periodic points
of $\phi$ is Zariski dense in $X$.

\end{pro}

\begin{proof}
Let $Y$ be the Zariski closure of non-critical periodic points of
$\phi$ in $X(\mathbb{F})$ and suppose $Y\neq X$. Let $q=p^n$,
$p=char(\mathbb{F})$, be such that $X$ as well as $\phi$ are defined
over the subfield $\mathbb{F}_q$ of $\mathbb{F}$ consisting of $q$
elements. Let $\sigma$ denote the Frobenius morphism of $X$ and let
$\Gamma_{\phi}$(resp. $\Gamma_m$) denote the graph of $\phi$ (resp.
$\sigma^m$) in $X\times X$. Let $U$ be an irreducible affine open
subset of $X-Y$ also defined over $\mathbb{F}_q$ such that $\phi$ is
an open embedding from $U$ to $X$ and let
$V=\Gamma_{\phi}\bigcap(U\times U)$. By Theorem \ref{hrushovski}
there exists $m>0$ such that
(V$\bigcap$$\Gamma_m$)($\mathbb{F}$)$\neq$$\emptyset$ i.e. there
exists $u\in U(\mathbb{F})$ such that $\phi(u)=\sigma^m(u)\in U$.
Since $\phi$ is defined over $\mathbb{F}_q$, it follows that $\phi^l(u)=\sigma^{lm}(u)\in U$. So
$\phi(u)$ is a non-critical periodic point of $\phi$. This
contradicts the definition of $Y$ and $U$, so the proof completes.

\end{proof}
For the convenience of the reader, we repeat the arguments of [Theorem 5.1, \cite{fa}] which allows us to lift any isolated periodic point from the special fiber to the generic fiber.
\begin{lem}\label{dvr}
Let $\mathbf{X}$ be a projective scheme, flat over a discrete valuation ring $R$ which
has fraction field $K$ and residue field $k$, $\Phi$  be a birational
map $\mathbf{X}\dashrightarrow \mathbf{X}$ over $R$ which is well
defined at least at a point on the special fiber. Let $X$ be the
special fiber of $\mathbf{X}$ and $X'$ be the generic fiber of
$\mathbf{X}$ , $\phi$ be the restriction of $\Phi$ to $X$, $\phi'$ be the
restriction of $\Phi$ to $X'$. If the set consisting of periodic
$\overline{k}$-points of $\phi$ is Zariski dense in $X$, and there
are only finitely many curves of periodic points, then the set consisting
of periodic $\overline{K}$-points of $\phi'$ is Zariski dense in the
generic fiber of $X'$.

\end{lem}

\begin{proof}
 The set of periodic $\overline{k}$-points of $\phi$ of period dividing
$n$ can be viewed as the set of $\overline{k}$-points in
$\Delta_X\bigcap\Gamma_{\phi^n}$, where $\Delta_X$ is the diagonal
and $\Gamma_{\phi^n}$ is the graph of $\phi^n$ in $X\times X$.

The set of non-critical periodic points of $\phi$ is Zariski dense
in $X$, thus the set of non-critical periodic points of $\phi$ not
in any curve of periodic points which are smooth points on $X$ is
Zariski dense. For any positive integer $n$, consider the subscheme
$\Delta_{\mathbf{X}}\bigcap\Gamma_{\Phi^n}$ of
$\mathbf{X}\times_R\mathbf{X}$, where $\Delta_{\mathbf{X}}$ is the
diagonal and $\Gamma_{\Phi^n}$ is the graph of $\Phi^n$ in
$\mathbf{X}\times_R \mathbf{X}$. If $x\in X$ is a periodic point of
$\phi$ not in any curve of periodic points which is also a smooth
point on $X$, then $(x,x)$ is contained in a closed subscheme of
$\Delta_{\mathbf{X}}\bigcap\Gamma_{\Phi^n}$ of dimension $1$. Since
$x$ is not in any curve of periodic points, the generic point $x'$
of this subscheme is in $\Delta_{X'}\bigcap\Gamma_{\phi'^n}$ the
generic fiber of $\Delta_{\mathbf{X}}\bigcap\Gamma_{\Phi^n}$. So
$x'$ is a periodic point. Since $x$ is non critical, $\Phi^k$ is a
locally isomorphism on a neighborhood $\mathbf{U}^k$ of $x$ in
$\mathbf{X}$. Since $x'$ is the generic point, $x'$ is also in all
of this $\mathbf{U}^k$. So it is non-critical.

We identify $\mathbf{X}$ with $\Delta_{\mathbf{X}}$. For an open
subset $U'$ of $X'$, let $Z'$ be a Weil-divisor of $X'$ containing
$X'-U'$. let $\mathbf{Z}$ be the closure of $Z'$ in $\mathbf{X}$,
then codim$\mathbf{Z}=1$ and each component of $\mathbf{Z}$ meets
$Z'$.  Each component of $X$ is of codimension 1. If
$X\subseteq\mathbf{Z}$, each component of $X$ is a component of
$\mathbf{Z}$. But $X'\bigcap X=\emptyset$, so
$X\nsubseteq\mathbf{Z}$. Let $\mathbf{U}=\mathbf{X}-\mathbf{Z}$ and
$U=\mathbf{U}\bigcap X$, then $\mathbf{U}\bigcap X'=U'$ and
$U\neq\emptyset$. There is a periodic point $x$ of $\phi$ not in any
curve of periodic points which is a smooth point in $U$, then $x'$
is in $U'$ and is a non-critical periodic point of $\phi'$.
\end{proof}

\subsection{Invariant curves}\label{InvariantC}

From the previous subsection, we see that curves of periodic points are the main obstructions to lift periodic points from finite field.  The following theorem of Cantat \cite{Invarianthypersurfaces} tells us that if $\lambda_1>1$ on the special fiber, then this obstruction can be removed.

\begin{thm}(\cite{Invarianthypersurfaces})\label{finite}
Let $X$ be a projective smooth surface. Then any birational map $f:X\rightarrow X$ with $\lambda_1(f)>1$ admits only finitely many periodic curves.

In particular, there are only finitely many curves of periodic points.
\end{thm}

\proof
We first recall the following theorem of Cantat.
\begin{thm}(\cite{Invarianthypersurfaces})\label{invarianthypersurface}
Let $X$ be a smooth  projective variety, and $f: X\dashrightarrow X$ be a birational
map. If there are infinitely many periodic hypersurfaces of $f$,
then $f$ preserves a non constant rational function (i.e. there is a
rational function $\Phi$ and a non zero constant $\alpha$ such that
$\Phi\circ f=\alpha\Phi$ ).

Moreover there is a proper modification $\pi:\widehat{X}\rightarrow X$ lifting $\Phi$ to a fibration $\widehat{\Phi}:
\widehat{X}\rightarrow \mathbb{P}^1$ and $f$ to a  birational
transformation $\widehat{f}$ of $\widehat{X}$ such that $\widehat{f}$ preserves the fibration $\widehat{\Phi}$.

\end{thm}
We conclude by the following well-known fact
\begin{pro}\label{lamda1=1}
Let $X$ be a projective smooth  surface, $f:X\dashrightarrow X$ be a birational map. If $f$
preserves a fibration $\Phi: X\rightarrow \mathbb{P}^1$, then
$\lambda_1(f)=1$.
\end{pro}
\proof[Proof of Proposition \ref{lamda1=1}]
By Theorem \ref{AS modele}, we may assume that $f$ is algebraically stable on $X$. Then there is a nef class $\omega\in N^1(X)_{\mathbb{R}}-\{0\}$, such that $f^*(\omega)=\lambda_1(f)\omega$. Let $[F]\in N^1(X)_{\mathbb{R}}$ be the class of a fiber. Since $f$ is birational, $f^*[F]=f_*[F]=[F]$. Then $$(\omega\cdot [F])=(\omega\cdot f_*[F])=(f^*\omega\cdot [F])=\lambda_1(f)(\omega\cdot [F]).$$ Then $(\omega\cdot [F])=0$ or $\lambda_1(f)=1$. If $(\omega\cdot [F])=0$, since $(F^2)=0$ and $L$ is nef, then $L=kF$ for some $k\in \mathbb{R}$ by Theorem \ref{Hodge}. Thus we also have that $\lambda_1(f)=1$, because $f^*[F]=[F]$.
\endproof

\endproof

\rem
In fact, this theorem is stated over $\mathbb{C}$ in \cite{Invarianthypersurfaces}. Its proof extends immediately to any field.
\endrem

\subsection{Proof of Theorem \ref{main1}}
We recall Theorem \ref{main1}.

\begin{thm}\label{main}(=Theorem \ref{main1})
Let $X$ be a projective surface over $\mathbf{k}$, and $f:X\dashrightarrow X$ be a
birational map. If
$\lambda_1(f)>1$ then the set of non-critical periodic points is
Zariski dense.

\end{thm}

\proof
Let $L$ be a very ample line bundle of $X$.
We may assume that the transcendence
degree of $\mathbf{k}$ over its prime field $F$ is finite, since we can
find a subfield of $\mathbf{k}$ which is finitely generated over $F$ such that $X$, $f$ and $L$ are all defined over this subfield.
We complete the proof by induction on the transcendence degree of $\mathbf{k}$ over $F$.

If $\mathbf{k}$ is the closure of a finite field, then the theorem holds by
Proposition \ref{Fakhruddin 5.5}.

If $\mathbf{k}=\overline{\mathbb{Q}}$, there is a regular subring $R$ of
$\overline{\mathbb{Q}}$ which is finitely generated over
$\mathbb{Z}$, such that $X$, $L$, $f$ are defined over $R$. By Theorem \ref{generally}, there is a maximal ideal $\mathfrak{m}$ of $R$ such that the fiber $X_m$ is smooth and the the restriction $f_m$ of $f$ on this fiber is birational with $\lambda_1(f_m)>1$. Since $R$ is regular
and finitely generated over $\mathbb{Z}$, $R_{\mathfrak{m}}$ the
localization of $R$ at $\mathfrak{m}$ is a d.v.r such that
$\overline{\Frac(R_{\mathfrak{m}})}=\overline{\mathbb{Q}}$ and
$R_{\mathfrak{m}}/\mathfrak{m}R_{\mathfrak{m}}=R/\mathfrak{m}$. Then
by Proposition \ref{Fakhruddin 5.5} the set of non-critical periodic
points of $f_m$ is Zariski dense in the special fiber.
Since $\lambda_1(f_m)>1$, there are only finitely many
curves of periodic points by Theorem \ref{finite}. The conditions in lemma
\ref{dvr} are all satisfied. Thus by lemma
\ref{dvr} the non-critical periodic points of $f$ form a Zariski dense
set. Then the theorem holds in this case.

If the transcendence degree of $\mathbf{k}$ over $F$ is greater than $1$, we pick an algebraically closed subfield $K$
of $\mathbf{k}$ such that the transcendence degree of $K$ over $F$ equals the transcendence degree of $\mathbf{k}$ over $F$ minus 1. Then we pick a subring $R$
of $k$ which is finitely generated over $K$, such that $X$, $L$ and $f$ are all
defined over $R$. Since $\Spec R$ is regular on an open set. We may assume that $R$ is regular by adding finitely many reciprocals of elements in $R$. We do the same argument as in the case $\mathbf{k}=\overline{\mathbb{Q}}$, then we prove the theorem.
\endproof
\subsection{Existence of Zariski dense orbits}\label{subam}
In this subsection, we denote by $\mathbf{k}$ an algebraically closed field of characteristic $0$.

\begin{thm}\label{coram}(=Theorem \ref{coram1})Let $X$ be a projective surface over an algebraically closed field $\mathbf{k}$ with characteristic $0$, $f:X\dashrightarrow X$ be a birational map with $\lambda_1(f)>1$. Then there is a point $x\in X$ such that $f^n(x)\in X-I(f)$ for any $n\in \mathbb{Z}$ and $\{f^{n}(x)|n\in \mathbb{Z}\}$ is Zariski dense.
\end{thm}

We prove this theorem in several steps. Let us first recall the following theorem of E. Amerik.
\begin{thm}\label{amerik}(\cite{Amerik}) Let $X$ be a variety over $\overline{\mathbb{Q}}$ and $f:X\dashrightarrow X$ be a birational map. Then there exists a point $x\in X$ such that $f^n(x)\in X-I(f)$ for any $n\in \mathbb{Z}$ and $\{f^{n}(x)|n\in \mathbb{Z}\}$ is infinite.
\end{thm}

We next extend this result to any algebraically closed field of characteristic $0$. To do so we shall need the following lemma.
\begin{lem}\label{section}Let $\pi:X\rightarrow Y$ be a morphism between two varieties defined over an algebraically closed field. For any point $x\in X$, there is an irreducible subvariety $S$ through $x$ of $X$, such that $\dim S=\dim Y$, and the restriction of $\pi$ on $S$ is dominant to $Y$.
\end{lem}

\proof
We prove this lemma by induction on  $\dim X-\dim Y$.

If  $\dim X-\dim Y=0$, we pick $S=X$. Then the lemma trivially holds.

If $ \dim X-\dim Y=n\geq 1$. We may assume that $X$ is affine. There is a projective completion $\widetilde{X}$ of $X$. Let $F$ be the fiber of $\pi$ through $x$. We have $\dim F\geq \dim X-\dim Y=n$. Pick $L$ a very ample line bundle of $\widetilde{X}$. Pick $\widetilde{X'}$ a general section of $L$ through $x$ such that $\widetilde{X'}$ intersects $\overline{F}$ properly in $\widetilde{X}$ by Bertini's Theorem (see \cite{Hartshorne1977}). Let $X'=\widetilde{X'}\bigcap X$ and $X''$ an irreducible component of $X'$ through $x$. Then $\dim X''-\dim Y=\dim X-1-\dim Y=n-1$ and for a general point $y\in Y$, the fiber $F_y$ at $y$ intersects $X''$ properly. So the morphism $\pi:X''\rightarrow Y$ is dominant. We conclude by the induction hypothesis.

\endproof

We are now in position to prove

\begin{pro}\label{am} Let $X$ be a variety over an algebraically closed field $\mathbf{k}$ of characteristic $0$, $f:X\dashrightarrow X$ a birational map, then there is a point $x\in X$ such that $f^n(x)\in X-I(f)$ for any $n\in \mathbb{Z}$ and $\{f^{n}(x)|n\in \mathbb{Z}\}$ is infinite.
\end{pro}
\proof
In the case $\mathbf{k}=\overline{\mathbb{Q}}$, it is just Theorem \ref{amerik}.

In the general case, there is a subring $R$ of $\mathbf{k}$ which is finitely generated over $\overline{\mathbb{Q}}$ on which $X$ and $f$ are defined. And we may assume that $\mathbf{k}$ is $\overline{\Frac(R)}$. Then we may pick a model $\pi:X_R\rightarrow \Spec R$ as a scheme over $R$ and $f_R$ as a birational map of $X_R$ over $R$ such that the geometric generic fiber of $X_R$ is $X$ and the restriction of $f_R$ on $X$ is $f$. Pick any $\mathfrak{m}\in R$ such that the restriction $f_{\mathfrak{m}}$ of $f$ on the special fiber $X_{\mathfrak{m}}$ at $\mathfrak{m}$ is birational. Since $R/\mathfrak{m}=\overline{\mathbb{Q}}$, by Theorem \ref{amerik}, there is a point $y\in X_{\mathfrak{m}}$ such that $f_{\mathfrak{m}}^n(y)\in X_{\mathfrak{m}}-I(f_{\mathfrak{m}})$ for any $n\in \mathbb{Z}$ and $\{f_{\mathfrak{m}}^{n}(y)|n\in \mathbb{Z}\}$ is infinite. By Lemma \ref{section}, we get an irreducible subvariety $S$ of $X$ such that $\dim S=\dim R$, and the restriction of $\pi$ to $S$ is dominant to $\Spec R$. Let $x$ be the generic point of $S$, then $x\in X_R(\overline{\Frac R})=X$. We have $f^n(x)\in X-I(f)$ for any $n\in \mathbb{Z}$ and $\{f^{n}(x)|n\in \mathbb{Z}\}$ is infinite.
\endproof

\begin{rem}Proposition \ref{am} can not be true over $\overline{\mathbb{F}_p}$, since for any $q=p^n$, $\# X(\mathbb{F}_q)$ is finite.

\end{rem}

\proof[Proof of Theorem \ref{coram}]
By Theorem \ref{invarianthypersurface}, there are only finitely many invariant curves of $f$. Let $C$ be the union of these curves and let $U=X-C$. Then by Corollary \ref{am}, there is a non-preperiodic point $x\in U$ such that $f^n(x)\in U-I(f)$ for any $n\in \mathbb{Z}$. Let $S$ be the Zariski closure of $\{f^{n}(x)|n\in \mathbb{Z}\}$ in $X$. If $S\neq X$, then $\dim S=1$ since it is infinite. Let $D$ be the union of curves in $S$, then $D$ is an invariant curve. Then $D\bigcap U\neq \emptyset$. It contradicts the definition of $U$.
\endproof

\section{Automorphisms  and  Saito's fixed point formula}\label{Automorphisms>1}
In this subsection, we denote by $\mathbf{k}$ an algebraically closed field of characteristic $0$.

\begin{thm}\label{cisop1}(=Theorem \ref{cisop}) Let $X$ be a projective smooth surface over an algebraically closed field $\mathbf{k}$ of characteristic $0$. And let $f:X\rightarrow X$ be a nontrivial $automorphism$ with $\lambda_1(f)>1$. Denote by $\#\Per_n(X)$ the number of isolated periodic points of period $n$ counted with multiplicities. Then, we have $$\#\Per_n(f)\sim \lambda_1(f)^n.$$
\end{thm}
This theorem follows from our next result.
\begin{thm}\label{autocount}Let $X$ be a smooth projective surface over $\mathbb{C}$ and $f$ be an automorphism of $X$ with $\lambda_1(f)>1$. Denote by $L(f^n):=\sum_i (-1)^i \Tr[f^{n*}:H^i(X)\rightarrow H^i(X)]$ the Lefschetz number of $f^n$. Then, we have  $$|\#\Per_n(X)-L(f^n)|=O(1).$$
\end{thm}

Observe that any birational map with $\lambda_1>1$ on a non-rational surface is birationally conjugated to an automorphism \cite{favre}.

Our proof of Theorem \ref{autocount} is based on a version of Lefschetz's formula due to Saito \cite{Saito1987} which was subsequently used by Iwasaki and Uehara in \cite{Iwasaki2010}.

\begin{rmk}
Theorem \ref{cisop1}  gives an alternative to Hrushovski Fakruddin's method to prove the Zariski density of periodic points for automorphisms. Indeed, let $Z$ be the Zariski closure of the set of all periodic points.  By the previous theorem, $Z$ is either a curve or $Z=X$. But $Z$ is $f$-invariant, hence $Z=X$ by Cantat's theorem.
\end{rmk}

\proof[Proof of Theorem \ref{cisop1}]
Suppose $\mathbf{k}=\mathbb{C}.$ Then by [Lemma 7.8,\cite{Iwasaki2010}], we have $L(f^n)\sim\lambda_1(f)^n.$ Since by Theorem \ref{autocount}, we have $$\#\Per_n(f)\sim L(f)^n,$$ it follows that $$\#\Per_n(f)\sim \lambda_1(f)^n,$$ as required.

In the general case, since there are only finitely many coefficients in the definition of $X$ and $f$, we may always assume that the transcendence degree of $\mathbf{k}$ over $\overline{\mathbb{Q}}$ is finite. We may then embedd $\mathbf{k}$ in $\mathbb{C}$, and apply the preceding argument.
\endproof

\subsection{Local invariants associated to fixed points}\label{sbslsf}
Let us  recall the ingredients appearing in Saito's formula. We refer to~\cite{Iwasaki2010} for more details.

We fix  $f:X\rightarrow X$ be an automorphism of a projective smooth  surface, and  denote by $\fix(f)$ the set of fixed points of $f$.

Pick any $x \in \fix(f)$, and write $\mathfrak{m}\subset \hat{\sO}_{X,x}$ for the maximal ideal
in the completion of the local ring at $x$. Since $X$ is smooth, $ \hat{\sO}_{X,x}\simeq \C [[z_1, z_2]]$ for some coordinates $z_1, z_2$, and $\mathfrak{m}$ is the ideal generated by $z_1, z_2$. We may then write
$$f(z_1, z_2)=( z_1+gh_1,  z_2+gh_2) $$
for some elements $g,h_1,h_2\in \mathfrak{m}$, where $g\neq 0$ and $h_1,h_2$ are relatively prime.  The first important invariant is
\begin{equation}\label{eq:def-delta}
\delta(f,x):=\dim_{\mathbb{C}}\hat{\sO}_{X,x} / f^* \mathfrak{m} =
\dim_{\mathbb{C}} \C[[z_1,z_2]] / (h_1,h_2).
\end{equation}
Observe that $ \delta$ is finite.

Next denote by $\Lambda(f,x)$ the irreducible components of the fixed point locus of $f$ at $x$, and pick $C \in \Lambda(f,x)$ any irreducible component.   Then we set
\begin{equation}\label{eq:def-vc}
v_C(f)=\ord_C(g)
\end{equation}
Note that
given  any \emph{reduced} equation  $h \in \mathfrak{m}$  of $C$, we have
$v_C(f)=\max\{m\in \mathbb{N}|\, h^m \text{ divides } g\} = \min\{ \ord_C ( \phi \circ f -\phi),\, \phi \in \mathfrak{m}\}$. In particular this quantity is independent on the choice of coordinates.

Let us introduce the holomorphic $1$-form $$\omega_{f,x}:=h_2dz_1-h_1dz_2.$$
A smooth curve $C\in \Lambda(f,x)$ is said to be:
\begin{itemize}
\item
of type I if the restriction $\omega_{f,x}|C \in \Omega^1_C$ is non zero;
\item
of type II if $\omega_{f,x}|C$ vanishes identically.
\end{itemize}
Observe that the form $\omega_{f,x}$ depends on the choice of coordinates, but  the type
of a curve does not. It is also independent on  the choice of a point $x$ on the curve.
We shall denote by  $X_\mathrm{I}(f)$ (resp. $X_\mathrm{II}(f)$) the set of irreducible curves
that are fixed by $f$ and of type I (resp. of type II).
Suppose $C\in \Lambda(f,x)$ is \emph{smooth}.
If it is of type I, then we set
\begin{equation}
\mu_{C,x}(f):=\ord_x(\omega_{f,x}|C).
\end{equation}
If $s:\mathbb{C}[[t]]\rightarrow C$ is any  local parametrization of $C$ at $x$, then we have $\mu_{C,x}(f)=\ord_t(s^*\omega_{f,x})$.
If $C$ is of type II, define
\begin{equation}
\mu_{C,x}(f):=\ord_x(\partial_{f,x}|C),
\end{equation}
where $\partial_{f,x} = h_1 \partial_{z_1} + h_2 \partial_{z_2}$. In terms of the $1$-form, $\omega_{f,x}$ this multiplicity can be interpreted  as follows. By assumption there exists $a\in \mathfrak{m}$ such that $a|C\neq 0$ and $\omega_{f,x}-adh$ is divisible by $h$. Then we have $\mu_{C,x}(f):=\ord_0(a)$. We leave to the reader to check that these quantities are independent on the choice of coordinates.

\begin{pro}\label{pro1}
Suppose $f: X \to X$ fixes a smooth
curve $C$ pointwise. Pick $x \in C$ and assume $df(x)$ has one eigenvalue $\lambda$ which is not a root of unity. Then $C$ is of type $I$, and we have
\begin{equation}\label{eq:easycase}
v_C(f^n) =1, \, \text{ and } \delta(f^n,x) =  \mu_{x,C}(f^n) =0
\end{equation}
for all $n \in \Z - \{ 0 \}$.
\end{pro}
\begin{proof}
Locally we may choose coordinates $z_1,z_2$ such that $C = \{ z_1 =0 \}$, and then we have
$ f (z_1, z_2) = ( z_ 1 + z_1 ( \lambda -1 + o(1)), z_2 + z_1 h_2)$ for some power series $ h_2$.
The computations of $v_C(f), \mu_{x,C}(f)$ and $\delta(f^n,x)$ then follow immediately from the definitions.
\end{proof}

\begin{pro}\label{pro2}
Suppose $f: X \to X$ fixes a point $x$, and assume $df(x) = \id$.
If the set of all periodic points has simple normal crossing singularities at $x$,
then we have
\begin{equation}\label{eq:id}
\delta(f^n,x)  =  \delta (f,x), \,
v_C(f^n) =  v_C(f), \text { and }
\mu_{x,C}(f^n) = \mu_{x,C}(f) ~,
\end{equation}
for all $n\in \Z - \{ 0 \}$.
\end{pro}
\begin{proof}
Since $f$ is tangent to the identity then there exists a unique formal vector field $\partial$ vanishing up to order $2$ and such that $\exp(\partial)=f$. Let us recall this vector field is constructed, see for instance~\cite{F.E.Brochero}.

 Choose coordinates $z_1, z_2$ and write $$f
 = (z_1 + g h_1, z_2 + gh_2) =
   (z_1 + \sum_{n\ge 2} p_n(z_1, z_2),
 z_2 + \sum_{n\ge 2} q_n(z_1, z_2))  $$ where $p_n, q_n$ are homogeneous polynomials of degree $n$. Similarly write $\partial = \sum_{n \ge 2}  a_n(z_1, z_2) \partial_{z_1} + \sum_{n \ge 2} b_n(z_1, z_2) \partial_{z_2}$ with $a_n, b_n$ homogeneous of degree $n$.
 For each $m\ge2$, set $\partial_m = \sum_{n \le m}  a_n(z_1, z_2) \partial_{z_1} + \sum_{n \le m} b_n(z_1, z_2) \partial_{z_2}$, and define recursively
 $\partial^j_m(\phi) = \partial_m ( \partial^{j-1}_m(\phi))$ for any $\phi$.
 Then we have
 \begin{eqnarray}
 p_{m+1} & =&
 a_{m+1}  + HT_{m+1} \left( \sum_{j=2}^m  \frac1{j!} \, \partial^j_m(z_1) \right)
 \label{eq1}\\
 q_{m+1} & =&
 b_{m+1}  + HT_{m+1} \left( \sum_{j=2}^m  \frac1{j!} \, \partial^j_m(z_2) \right) ~,
 \label{eq2}
 \end{eqnarray}
where $HT_{m+1} (\phi)$ denotes the homogeneous part of degree $m+1$ of the power series expansion of $\phi$ in $z_1, z_2$.

Since the fixed point locus of $f$ is assumed to have simple normal crossing singularities at $x$  we may choose coordinates such that $g(z_1, z_2) = z_1^{n_1} z_2^{n_2}$ for some $n_1, n_2 \ge0$ with $n_1 + n_2 \ge 1$.

We first claim that $\partial =  z_1^{n_1} z_2^{n_2} \, \tilde{\partial}$ for some reduced formal vector field $\tilde{\partial}$, i.e. whose zero locus is zero dimensional.

Indeed by assumption $z_1^{n_1} z_2^{n_2}$ divides $p_n$ and $q_n$ for all $n \ge2$. Let us prove by induction that $z_1^{n_1} z_2^{n_2}$ divides
$a_n$ and $b_n$ for all $n$. This is true for $n=2$ since $p_2 = a_2$ and $q_2 = b_2$.
Suppose it is true for all $m \le n$. Then $z_1^{n_1} z_2^{n_2}$ divides $\partial_m$  for $2 \le m \le n$, hence $\partial^j_m(z_1)$ and $\partial^j_m(z_2)$ for all $j$, and it follows from~\eqref{eq1} and~\eqref{eq2} that $z_1^{n_1} z_2^{n_2}$ also divides $a_{n+1}$ and $b_{n+1}$ as required. Conversely if for some $m_1,m_2$ the monomial $z_1^{m_1}z_2^{m_2}$ divides $a_n$, $b_n$ for all $n$, the same argument shows it divides $p_n$ and $q_n$ as well for all $n$. This proves the claim.

The claim implies by definition that $v_C(f) = \ord_C ( \partial)$ for any curve of fixed point $C$ of $f$. Since $f^n = \exp ( n \partial)$ by construction, it follows that
$$
v_C(f^n) = \ord_C(n \partial) = \ord_C(f) = v_C(f)~.
$$
Assume now $C = \{ z_1 =0 \}$ is a curve a fixed point so that $n_1 \ge1$.
Write $\tilde{\partial} = \tilde{a} \partial_{z_1} + \tilde{b} \partial_{z_2}$ with $\tilde{a}, \tilde{b}$ having no common factors.

Suppose first $\tilde{\partial}$ is generically transverse to $C$, i.e. $\tilde{a}(0,z_2) \not \equiv 0$.
Let us compute
$\exp( \partial)(z_1) = z_1 + \sum_{j \ge 1} \frac1{j!}\, \partial^j z_1$.
Write $\partial^j z_1 = g \tilde{a}_j$ so that
\begin{equation}\label{eqbt}
\tilde{a}_1 = \tilde{a} \text{ and }
\tilde{a}_{j+1} = \tilde{\partial} g \, \tilde{a}_j + g\, \tilde{\partial} \tilde{a_j}.
\end{equation}
Then for any $j \ge1$, we get
$$
\ord_0( \tilde{a}_{j+1} (0,z_2))
=\ord_0( \tilde{\partial} g (0,z_2)) + \ord_0( \tilde{a}_{j} (0,z_2))
\ge 1 + \ord_0(\tilde{a})~.
$$
Since $ f(z_1, z_2) = (z_1 + g \sum \frac1{j!}\, \tilde{a}_j,\star)$,  we conclude first that $C$ is of type I and then that $\mu_{x,C}(f) = \ord_0 (\sum \frac1{j!}\, \tilde{a}_j (0,z_2)) = \ord_0( \tilde{a}(0,z_2))$.
This proves that
$$\mu_{x,C}(f^n) = \ord_0( n \tilde{a}(0,z_2)) = \ord_0( \tilde{a}(0,z_2)) = \mu_{x,C}(f)~.$$
Suppose next that $C$ is $\tilde{\partial}$-invariant, i.e.
$\tilde{a}(0,z_2)  \equiv 0$ but $\tilde{b}(0,z_2) \not\equiv 0$. We are now interested in
$\exp( \partial)(z_2) = z_2 + \sum_{j \ge 1} \frac1{j!}\, \partial^j z_2$.
Write $\partial^j z_2 = g \tilde{b}_j$ so that as before we have
\begin{equation}\label{eqat}
\tilde{b}_1 = \tilde{b} \text{ and }
\tilde{b}_{j+1} = \tilde{\partial} g \, \tilde{b}_j + g\, \tilde{\partial} \tilde{b_j}.
\end{equation}
Then it is not difficult to see that $C$ is of type II, and
$ \ord_0( \tilde{b}_{j+1} (0,z_2))\ge 1 + \ord_0(\tilde{b})$ for all $j \ge 1$, so that
$$\mu_{x,C}(f^n) = \ord_0( n \tilde{b}(0,z_2)) = \ord_0( \tilde{b}(0,z_2)) = \mu_{x,C}(f)~.$$
Finally let $I =\langle \tilde{a}, \tilde{b} \rangle \subset \hat{\sO}_{X,x}$ be the ideal generated by $\tilde{a}, \tilde{b}$.
Since $\tilde{\partial} \phi \in I$ for any $\phi$, by induction on $j$ we see that
$\tilde{b}_{j+1}, \tilde{a}_{j+1} \in I^2 + (g) I \subset \mathfrak{m} \cdot I \subset I$ for all $j \ge 1$.
From the identities $h_1 = \sum \frac1{j!}\, \tilde{a}_j$ and $h_2= \sum \frac1{j!}\, \tilde{b}_j$, we infer $J := \langle h_1, h_2\rangle \subset I$.

We claim that the integral closures of $I$ and $J$ are equal. Grant this claim. For any $\mathfrak{m}$-primary ideal  $\mathfrak{a} \subset \hat{\sO}_{X,x}$, we let
$e(\mathfrak{a}) = \lim_{n\to \infty} \frac1{2 n^2} \dim_{\{\mathbb{C}\}}( \hat{\sO}_{X,x}/ \mathfrak{a}^n ) $ be the (Hilbert-Samuel) multiplicity of $\mathfrak{a}$. Two ideals having the same integral closure have the same multiplicity, see~\cite{lejeune-teissier}. We thus have
$$
\delta(x,f) : = e(J) = e(I) = e  \langle \tilde{a}, \tilde{b} \rangle =   e   \langle n\tilde{a},n \tilde{b}\rangle = \delta(f^n,x)
$$
for all $n \neq 0$,
which concludes the proof of the proposition.

To prove the claim,  pick any sequence of point blow-ups $ \pi: \hat{X} \to  X$ centered above $x$ and such that the ideal sheaf $J \cdot \sO_{\hat{X}} $ is locally principal so that we can write $J \cdot \sO_{\hat{X}} =\sO_{\hat{X}} ( - \sum m_i E_i)$ where $E_i$ are exceptional  and $m_i = \ord_{E_i} (\pi^*J) \ge 1$.
Now recall that  $h_1 - \tilde{a}$  and $h_2 - \tilde{b}$ lie
in $\mathfrak{m} \cdot I$. Pick any exceptional curve $E_i$. By definition $\ord_{E_i} ( \pi^* I )
= \min\{ \ord_{E_i} ( \pi^* \tilde{a}), \ord_{E_i} ( \pi^* \tilde{b})\}$. Say
$\ord_{E_i} ( \pi^* I ) = \ord_{E_i} ( \pi^* \tilde{a})$. Then
$$\ord_{E_i} (\tilde{a} - h_1) \ge \ord_{E_i} ( \pi^* \mathfrak{m}) +
\ord_{E_i} ( \pi^* I)
> \ord_{E_i} ( \pi^* \tilde{a})$$ hence $\ord_{E_i} (\pi^* h_1)   =  \ord_{E_i} (\pi^* \tilde{a}) =\ord_{E_i} ( \pi^* I )$.
On the other hand, $\ord_{E_i}( \pi^* h_2) \ge \ord_{E_i}( \pi^* I)$, hence
$\ord_{E_i}( \pi^*J) = \ord_{E_i}( \pi^* I)$. It follows from~\cite[Th\'eor\`eme 2.1 (iv)]{lejeune-teissier} that $I$ is included in the integral closure $\bar{J}$ of $J$, and $J \subset I \subset \bar{J}$ implies $\bar{I} = \bar{J}$ as was to be shown.
\end{proof}
\begin{rmk}
In the tangent to the identity case the proof shows we have the following geometrical interpretations. Let $\mathcal{F}$ be the (formal) foliation associated to
the formal vector field $\partial$ vanishing up to order $2$ at $0$ and satisfying $ f = \exp (\partial)$. Let $ \tilde{\partial}$ be the reduced vector field associated to $\partial$.

Then a curve $C$ of fixed points is of type I, if it is generically transversal to $\mathcal{F}$,
and of type II if it is a leaf of $\mathcal{F}$. The multiplicity $v_C(f)$ is the generic order of vanishing of $\partial$ along $C$; $ \delta(f)$ is the Hilbert-Samuel multiplicity of the ideal generated by $\tilde{\partial} \mathfrak{m}$. When $C$ is smooth, then
$\mu_{x,C}(f)$ is the order of vanishing of $\tilde{\partial}|C$ (when $C$ is of type II),
or of  its dual $1$-form (when $C$ is of type I).
\end{rmk}

Finally we define
\begin{equation}
v_{(C,x)}(f):=\delta(f,x)+\sum_{C\in\Lambda(f,x)}v_{C}(f)\, \mu_{C,x}(f).
\end{equation}

\subsection{Proof of Theorem \ref{autocount}}

By Theorem \ref{finite}, there are only finitely many curves of periodic points on $X$.
In particular, one can find an integer $M$ such that any curve of periodic points in included in $\fix(f^M)$. In the sequel we shall assume that $M=1$.
Pick any curve $C \in \fix_1(f)$.  At any point $x\in C$, the differential $dF_x$ has two eigenvalues, one equal to $1$ and the other to $\lambda(x) = \det dF_x$. Since $C$ is compact, it follows that $\lambda(x)  \equiv \lambda(C)$ is a constant.
Up to replace $f$ by a suitable iterate, we may also assume that either $\lambda(C) =1$ or $\lambda(C)^n \neq 1$ for all $n\ge1$.

\proof[Step 1:] Suppose that $\fix(f)$ has only simple normal crossing singularities.
We apply Saito's formula, see~\cite[Theorem 1.2]{Iwasaki2010}, and use results from the preceding section.
\begin{thm}
Assume $f: X \to X$ is an automorphism such that all irreducible components of $\fix(f^n)$ are smooth.
Then we have
\begin{equation}\label{eqa:saito}
L(f^n)=\sum_{x\in \fix(f^n)}v_x(f^n)+\sum_{C\in X_\mathrm{I}(f^n)}\chi(C)\, v_C(f^n)+\sum_{C\in X_\mathrm{II}(f^n)}(C^2)\, v_C(f^n).
\end{equation}
\end{thm}
Here $\chi(C)$ denotes the Euler characteristic of $C$, and $C^2$ its self-intersection.
It follows easily from our standing assumptions and from Propositions~\ref{pro1} and~\ref{pro2} that $|L(f^n) - \# \Per_n|$ is actually independent on $n$.

\proof[Step 2:] Let $S_1\subseteq X$ be the set of  singular points of curves of periodic points. This set is finite and $f$-invariant so that $f$ lifts as an automorphism $f_1$ to the blowup $\pi_1: X_1\rightarrow X$ of $X$ at all points in $S_1$. The exceptional components of $\pi_1$ are permuted by $f_1$, and we  have $$|\Tr(f^{*n}_1)_{H^{1,1}}-\Tr(f^{*n})_{H^{1,1}}|\leq \#S_1.$$ On the other hand, $\pi_*$ induces an isomorphism for $(i,j)\neq (1,1)$. So $$|L(f_1^n)-L(f^n)|=O(1).$$ There are at most $2$ isolated fixed points of $f^n$ on each exceptional component, hence $|\# \Per_n(X_1)-\# \Per_n(X)|=O(1)$. Repeating the argument finitely many times, we end up with an automorphism for which the union of all curve of periodic points has only simple normal crossing singularities. This concludes the proof.

\section{The case $\lambda_1=1$}\label{case2}
In this section, we denote by $\mathbf{k}$ an algebraically closed field of characteristic different from $2$ and $3$.

To conclude Theorem\ref{class}, we consider the case $\lambda_1=1$.
\begin{thm}(\cite{favre}, \cite{Gizatullin1980})\label{fadi}
Let $X$ be a projective smooth surface over $\mathbf{k}$, $L$ be an ample line bundle on $X$, and $f$ be a birational map of $X$. Assume $\lambda_1(f)=1$ then we are in one of the following  four cases.
\begin{points}
\item $\deg_L(f^n)$ is bounded. In this case, there is a birational model $X'$ of $X$ on which $f$ induces an automorphism and some iterate of $f$ acts on $N^1(X)$ as identity.
\item $\deg_L(f^n)\sim cn$ with $c>0$. In this case,  for some $n>0$, $f^n$ is birationally conjugated to a map of the form
$$f^n:(x,y)\in C\times \mathbb{P}^1\dashrightarrow \Big(g(x),\frac{A_1(x)y+B_1(x)}{A_2(x)y+B_2(x)}\Big),$$ where $C$ is a smooth curve, $g$ is an automorphism of $C$ and $A_1(x),B_1(x),A_2(x),B_2(x)$ are rational functions on $C$, such that $A_1(x)B_2(x)-A_2(x)B_1(x)\neq 0$.
\item $\deg_L(f^n)\sim cn^2,$ with $c>0$. In this case, there is a birational model $X'$ of $X$ and an elliptic fibration $$\pi: X'\rightarrow C$$ where $C$ is a smooth curve, such that $f$ is an automorphism on this model which preserves the fibration $\pi$.

\end{points}
\end{thm}

We consider these maps case by case.
\subsection{The case $\deg(f^n)$ is bounded}
\begin{pro}\label{pic0}Let $X$ be a projective variety, $f$ be an automorphism of $X$ which acts on $N^1(X)$ as identity. If the periodic points of $f$ are Zariski dense, Then there is an integer $N>0$ such that $f^N=\id$.
\end{pro}

\proof We denote $\Aut(X)$ the automorphism group of $X$. Let $\omega\in N^1(X)$ be an ample class, we denote $\Aut_{\omega}(X)=\{g\in \Aut(X)| g^*\omega=\omega \}$ be the subgroup of $\Aut(X)$. Let $L\in \Pic(X)$ such that $[L]=\omega \in N^1(X)$. For any $g\in \Aut_{\omega}(X)$, let $\Gamma_g\subseteq X\times X$ be the graph of $g$. We denote $\pi_1$, $\pi_2$ the projections to the first and second factors. Then $\pi_1^*L\otimes\pi_2^*L$ is ample on $X\times X$. We may consider the Hilbert polynomial $P_g(n)$ of $\Gamma_g$, $$P_g(n)=\chi(\Gamma_g, (\pi_1^*L\otimes\pi_2^*L)^{\otimes n})=\chi(X, (L\otimes g^*L)^{\otimes n}).$$ By Hirzebruch-Riemann-Roch theorem, we see that $\chi(X, (L\otimes g^*L)^{\otimes n})$ is a polynomial of $n$ whose coefficients only depend on the numerical class $[L\otimes g^*L]=2\omega \in N^1(X)$. So $P_g$ is independent of $g$, we may denote it by $P$. Let $Y$ be the Hilbert Scheme of $X\times X$ with Hilbert polynomial $P$. We see that $Y$ is of finite type and $\Aut_{\omega}(X)$ is an open subvariety of $Y$. Let $\Aut_{0}(X)$ be the irreducible component of $\Aut_{\omega}(X)$ containing $\id$. We suppose that $\Aut_{\omega}(X)$ has $M$ components, then for any $g\in \Aut_{\omega}(X)$,  $g^M\in \Aut_{0}(X)$. Since $f$ acts on $N^1(X)$ as identity, $f\in \Aut_{\omega}(X)$. Then $f^M\in \Aut_{0}(X)$. We may assume that $f\in \Aut_{0}(X)$. Let $S_n$ be the set of fixed points of $f^{n!}$ for $n\geq 1$. Then $S_n\subseteq S_{n+1}$ for any $n\geq 1$.  Let $F_n=\{g\in \Aut_{0}(X)|g_{|S_n}=\id\}$. Then $F_n$ is closed and $F_{n+1}\subseteq F_n$ for any $n\geq 1$. So there is an integer $l$ such that $F_{l}=\bigcap_{n\geq 1}F_n.$ Let $N=l!$, then $f^N\in F_{l}=\bigcap_{n\geq 1}F_n$. So $f^N_{|S_n}=\id$ for any $n\geq 1.$ Since the Zariski closure of $\bigcup_{n\geq 1} S_n$ is $X$, we conclude $f^N=\id$.

\endproof

\begin{pro}\label{cbounded}
Let $X$ be a projective smooth surface over $\mathbf{k}$, $L$ be an ample line bundle on $X$ and $f$ be a birational map of $X$, such that $\deg_L(f)$ is bounded. If the non critical periodic points of $f$ are Zariski dense, then there is an integer $n>0$ such that $f^n=\id$.

\end{pro}

\proof
By Theorem \ref{fadi}, we may assume that $f$ is an isomorphism and acts on $N^1(X)$ as identity. Then we conclude that there is an integer $n>0$ such that $f^n=\id$ by Proposition \ref{pic0}.
\endproof

\subsection{The linear growth case}

\begin{pro}\label{linearg}
Let $X$ be a projective smooth surface over $\mathbf{k}$, $L$ be an ample line bundle on $X$ and $f:X\rightarrow X$ be a birational map on $X$, such that $\deg_L(f^n)=cn+O(1),$ where $c>0$ . Then the non-critical periodic points are Zariski dense if and only if for some $n>0$, $f^n$ is birationally conjugated to a map of the form
$$f^n:(x,y)\in C\times \mathbb{P}^1\dashrightarrow \Big(x,\frac{A_1(x)y+B_1(x)}{A_2(x)y+B_2(x)}\Big)\,\,\,\,\,\,\,\,\,\,\,\,\,\,(*),$$ where $C$ is a smooth curve and $A_1(x),B_1(x),A_2(x),B_2(x)$ are rational functions on $C$, such that $A_1(x)B_2(x)-A_2(x)B_1(x)\neq 0$.

\end{pro}
\proof
Suppose first the set of non-critical periodic points is Zariski dense.
By Theorem \ref{fadi}, we may assume that $$f=\Big(g(x),\frac{A_1(x)y+B_1(x)}{A_2(x)y+B_2(x)}\Big),$$ where $g$ is an automorphism of $C$ and $A_1(x),B_1(x),A_2(x),B_2(x)$ are rational functions on $C$, such that $A_1(x)B_2(x)-A_2(x)B_1(x)\neq 0$.  Since the non-critical periodic points of $f$ are Zariski dense, then the periodic points of $g$ are Zariski dense. But $g$ is an automorphism of a projective curve, so there is an integer $n$, such that $g^n=id$. Replacing $f$ by a suitable iterate, we may assume that $g=\id$, then $$f=\Big(x,\frac{A_1(x)y+B_1(x)}{A_2(x)y+B_2(x)}\Big).$$

Conversely suppose $\deg_L(f^n)\rightarrow \infty$ and $f$ can be written under the form (*). We denote the function field of $C$ by $K$. Let $$T(x)=(A_1(x)+B_2(x))^2/(A_1(x)B_2(x)-A_2(x)B_1(x))$$ and let $t_1,t_2\in \overline{K}$ be the two eigenvalues of the matrix
$$
\left(\begin{array}{ccccccccc}
A_1(x) & B_1(x) \\
A_2(x) & B_2(x)
\end{array}\right).$$
If $(A_1(x)+B_2(x))^2/(A_1(x)B_2(x)-A_2(x)B_1(x))\in \mathbf{k}$,
then $$t_1/t_2+t_2/t_1+2=(A_1(x)+B_2(x))^2/(A_1(x)B_2(x)-A_2(x)B_1(x))\in \mathbf{k}.$$ Then $t_1/t_2\in \mathbf{k}$, since $\mathbf{k}$ is algebraically closed.

If $t_1=t_2$, then $t_1=t_2=(A_1(x)+B_2(x))/2\in K$. So we may replace $A_i(x),B_i(x)$ by $2A_i(x)/(A_1(x)+B_2(x)),2B_i(x)/(A_1(x)+B_2(x))$. So we may assume that $t_1=t_2=1$. Thus by changing the coordinate $f$ can be written as $(x,y+B(x))$ where $B(x)\in K$. Then for any $k>0$ $f^k=(x,y+kB(x))$. Then $\deg_L f^k$ is bounded. It is a contradiction.

If $t_1\neq t_2$, then $L=K(t_1)$ is a finite field extension over $K$. There is a curve $B\rightarrow C$ corresponding to this field extension. Since $f$ acts on the base curve $C$ trivially, it induces a map $\widetilde{f}$ on $\mathbb{P}^1\times_{C}B$. We see that the non-critical periodic points are also dense. Since $t_1, t_2$ are rational functions on $B$, by changing the coordinate $\widetilde{f}$ can be written as $(x,(t_1/t_2)y)$. Then $\deg_L f^k$ is bounded. It is a contradiction.

Then $T(x)$ is non-constant, for any $n>0$, pick a primitive root $r_n$ of unit with order $n$. Then there is at least one point $x\in C$ such that $T(x)=2r_n+2$. Then in some coordinate, $f$ acts on this fiber as $y\rightarrow r_n y$. So $f$ has $n$ periodic points on this fiber.  So the periodic points of $f$ are Zariski dense.

\endproof

\subsection{The quadratic growth case}

\begin{pro}\label{n^2}
Let $X$ be a projective smooth surface over $\mathbf{k}$ and an  elliptic fibration $$\pi: X\rightarrow C$$ where $C$ is a smooth curve.  Let $L$ be an ample line bundle on $X$ and $f$ be an automorphism of $X$ preserving this fibration such that $\deg_L(f^n)\rightarrow \infty$. Then the periodic points of $f$ are Zariski dense if and only if $f^N$ acts on $C$ as identity for some $N>0$.
\end{pro}

\proof
Let $\overline{f}$ be the map of the base space $C$ induced by $f$.

Suppose first that the periodic points of $f$ are Zariski dense. Then the periodic points of $\overline{f}$ are Zariski dense also. So there is an integer $N>0$, such that $\overline{f}^N=\id$.

Conversely suppose that $\deg_L(f^n)\rightarrow \infty$ and $\overline{f}=\id$. Since $f$ is an elliptic fibration,  then all but finitely many fibers are elliptic curves.

By [Theorem 10.1 III, \cite{Silverman1986}], the automorphism group of an elliptic curve (as an algebraic group) has order
divides $24$
so that we may replace $f$ by $f^{24}$. Then the restriction of $f$ to any smooth fiber is a translation. On any elliptic fiber $E_x$, if $f_{|E_x}$ has a fixed point if and only if $f_{|E_x}$ is identity.

We assume that the periodic points of $f$ are not Zariski dense. Then there is a set $T=\{x_1,\cdots,x_m\}\in \mathbb{P}^1$, such that for any $x\in C-T$, $E_x$ is a smooth elliptic curve and $f$ has no periodic points in $E_x.$

 By replacing $L$ by $L^n$ for $n$ sufficiently large, we may assume that $L$ is very ample. By Bertini's Theorem (see \cite{Hartshorne1977}),  we can find a general section $S$ of $L$ such that for any $y\in T$, the intersection of $S$ and $F_y$ is transverse, and these intersection points are smooth in $S$ and in $F_y$.

   Let $H$ be the set of periodic points which lies in $S$. Then $H\subseteq S\bigcap (\bigcup_{x\in T}E_x)$ is a finite set.  By replacing $f$ by $f^l$ for some $l$, we may assume that any $x\in H$ is a fixed point. Let $S^n=f^n(S)$ for $n\geq 0$. Then for any $ i\neq j$, $S^i\bigcap S^j=H$. Let $x$ be any point in $H$, $\pi(x)=y$. Since $S$, $F_y$ are smooth at $x$ and the intersection of $S$ and $E_y$ at $x$ is transverse.
 Then let $A_x$ be the completion of the local ring at $x$, we may identify $A_x$ to $\mathbf{k}[[z_1,z_2]]$. In some coordinate  we may assume that $S =(z_2=0)$ and $E_y=(z_1=0)$ in local. Then $f^{-1}$ can be written as
 $$f^{-1}=(z_1,z_2+h(z_1,z_2))$$
 where $h(0,0)=0$. Since $h(z_1,z_2)\neq 0$ and $h(0,0)=0$, $h(z_1,z_2)=0$ defines a curve $D$ in local and any points in $D$ is fixed by $f^{-1}$. Then $x$ is not isolate. But there is an open set $U$ such that there is no fixed point in $U-E_y$, then $h$ has the form $z_1^i(a+b(z_1,z_2))$ where $i\geq 1$, $a\neq 0$ and $b(0,0)=0$.

 \begin{lem}\label{n^22}
One can find local coordinates $(z_1, z_2)$ such that for any integer $n$, one has
$$f^{-n}=(z_1,z_2+z_1^i(na+b_n(z_1,z_2))) $$
 where $b_1=b$ and $b_n(0,0)=0$.
 \end{lem}

\proof[Proof of Lemma \ref{n^22}]
We proceed by induction.
If it is true for $n$, then we have $$f^{-(n+1)}=f^{-1}\circ f^{-n}.$$ So
$$f^{-n-1}=(z_1,z_2+z_1^i(na+b_n(z_1,z_2))+z_1^i(a+b(z_1,z_2+z_1^i(na+b_n(z_1,z_2))))))$$
$$=(z_1,z_2+z_1^i((n+1)a+b_{n+1}(z_1,z_2))) $$
where $b_{n+1}=b_n(z_1,z_2)+b(z_1,z_2+z_1^i(na+b_n(z_1,z_2)))$. Then $b_{n+1}(0,0)=0$.

\endproof

For $n$ not divided by the characteristic of $\mathbf{k}$, we can compute $S^n$ in local coordinates. By the lemma above, we know that $S^n$ is defined by the local equation $z_2+z_1^i(na+b_n(z_1,z_2))=0$. So all the intersections of $S$ and $S^n$ at $x$ are $i$. This $i$ is just dependent of $x\in H$ and independent of $n$. We denotes it by $i_x$, then $$(f^*L\cdot L)=(S\cdot S^n)=\Sigma_{x\in H}i_x$$ which is independent of $n$. It contradicts the fact that $\deg_L(f^n)=(f^*L\cdot L)\rightarrow\infty$. So the periodic points of $f$ are Zariski dense.

\bibliography{dd}

\end{document}